\documentclass[11pt]{article}
\usepackage{amsfonts,amssymb,amsthm,amsmath,latexsym,bm,enumerate,graphicx,subfigure,epsfig,epstopdf,color,caption}

\usepackage{verbatim}

\usepackage[T1]{fontenc}
\usepackage{authblk}

\usepackage[colorlinks,linktocpage,linkcolor=blue]{hyperref}
\usepackage{fancyhdr}
\usepackage{ulem}

\allowdisplaybreaks
\numberwithin{equation}{section}
\newtheorem{thm}{Theorem}[section]
\newtheorem{lem}{Lemma}[section]
\newtheorem{cor}{Corollary}[section]

\newtheorem{problem}{Problem}[section]

\newtheorem{example}{Example}[section]

\begin{document}

\title{H\"{o}lder Stable Recovery of the Source in Space-Time Fractional Wave Equations}
\author[1]{Kuang Huang\thanks{kuanghuang@cuhk.edu.hk}}
\affil[1]{\normalsize{Department of Mathematics, The Chinese University of Hong Kong, Hong Kong}}

\author[2]{Zhiyuan Li\thanks{lizhiyuan@nbu.edu.cn}}
\affil[2]{\normalsize{School of Mathematics and Statistics, Ningbo University, China}}

\author[3,4]{Zhidong Zhang\thanks{zhangzhidong@mail.sysu.edu.cn}}
\affil[3]{\normalsize{School of Mathematics (Zhuhai), Sun Yat-sen University, Zhuhai 519082, Guangdong, China}}
\affil[4]{\normalsize{Guangdong Province Key Laboratory of Computational Science, Sun Yat-sen
University, Guangzhou 510000, Guangdong, China}}

\author[5]{Zhi Zhou\thanks{zhizhou@polyu.edu.hk}}
\affil[5]{\normalsize{Department of Applied Mathematics, The Hong Kong Polytechnic University,  Hong Kong.}} 

\maketitle

\textbf{Abstract:}  We study the recovery of a spatially dependent source in a one-dimensional space-time fractional wave equation using boundary measurement data collected at a single endpoint. The main challenge arises from the fact that the eigenfunctions of the Dirichlet eigenvalue problem do not form an orthogonal system, due to the presence of a fractional derivative in space. To address this difficulty, we introduce a bi-orthogonal basis for the Mittag--Leffler functions and use it to establish uniqueness and H\"older-type stability results, provided the measurement time is sufficiently large. A Tikhonov regularization method is then employed to numerically solve the inverse source problem. Several numerical examples are presented to demonstrate the accuracy and efficiency of the proposed method and to validate our theoretical findings.\vskip5pt


\section{Introduction and main results}

In this paper, letting $T>0$ and $\alpha,\beta\in (1,2)$ be fixed constants,  we will investigate the following initial-boundary value problem of time and space fractional wave equation:
\begin{equation}
\label{eq-gov}
\left\{
\begin{aligned}
&\left( \partial_{t}^{\alpha} - D_x^\beta \right) u(x,t)=\lambda(t) f(x), \quad (x,t)\in (0,1)\times(0,T],\\
&u(0,t)=u(1,t)=0, \quad t\in(0,T],\\
&u(x,0)=u_t(x,0)=0,\quad x\in(0,1),
\end{aligned}
\right.
\end{equation}
where $\partial_{t}^\alpha$ is the Caputo left-sided fractional derivative of order $\alpha\in(1, 2)$ defined by
\begin{equation*}
\partial_t^\alpha \psi(t)=\frac1{\Gamma(2-\alpha)}\int_0^t (t-\tau)^{1-\alpha} \psi''(\tau) d\tau,\quad 1<\alpha<2,
\end{equation*}
and by $D_x^\beta$ we denote the non-self adjoint operator of order $\beta\in(1, 2)$:
\begin{equation*}
D_x^\beta \varphi(x)=\frac1{\Gamma(2-\beta)} \frac{d}{dx}\int_0^x (x-\xi)^{1-\beta} \varphi'(\xi) d\xi,\quad 1<\beta<2,
\end{equation*}
in which $\Gamma(\cdot)$ denotes the Gamma function, and we refer to Podlubny \cite{Podlubny-1999} for more properties of the above fractional derivatives.

The governing equation in \eqref{eq-gov} is a time and space fractional diffusion equation of super-diffusion type. Owing to the applicability in describing the anomalous phenomena in various fields, there are many applications of time or space or space-time differential equations in various fields, including biology \cite{ionescu2017role, magin2010fractional}, physics \cite{del2004fractional, kulish2002application, povstenko2015linear}, chemical \cite{singh2017analysis}, finance \cite{tenreiro2016relative}, viscoelasticity processes \cite{mainardi2022fractional}.
The broad applicability of fractional wave equations has also sparked significant interest among mathematicians. For notable theoretical contributions, see, for example, \cite{gorenflo2000wright, Gorska:2025, Han:2020, kochubei2014asymptotic, kubica2020time, luchko2019subordination, sakamoto2011initial}.  For an extensive and up-to-date overview of numerical methods for fractional models, we recommend the monograph \cite{Handbook:2019,JinZhou:2023book}, as well as the survey articles \cite{Bonito:2018, Du:ACTA, JinLazarovZhou:overview, Lischke:2020} and the references therein.

On the basis of the aforementioned foundation of direct problems, there are tremendously many works on the inverse source problems for fractional partial differential equations, which have practical significance in fields such as medical imaging, environmental monitoring, and material science. 
In this paper, we assume $\lambda(t)$ is given and we use additional information at one endpoint to carry out the inversion of the space component in the source term. In particular, we are concerned with the following inverse problem:
\begin{problem}\label{prob-isp}
Let $\lambda\in C[0,T]$ be known with $\lambda(0)\neq0$.  We propose:
\begin{enumerate}
\item (Uniqueness) Given the boundary observation $\phi(t):=u_x(1,t)$, $t\in(0,T)$. We investigate whether the measurement uniquely determines the source term $f(x)$, $x\in\Omega$.
\item (Stability) Establish a (conditional) stability for the mapping $u_x(1,\cdot) \mapsto f$?
\item (Reconstruction) Design an efficient algorithm to recover the unknown source profile $f$.
\end{enumerate}
\end{problem}
Most existing studies on inverse source problems with boundary measurements have focused on the case where $0 < \alpha < 2$ and $\beta = 2$. For example, Zhang and Xu \cite{zhang-xu-2011} investigated an inverse source problem for a fractional diffusion equation with $\alpha = 1/2$ and $\beta = 2$, employing refined Carleman estimates. 
In \cite{KianYamamoto:2019}, Kian and Yamamoto studied the stable determination of a space-time dependent source term in (fractional) diffusion equations ($\alpha \in (0,1]$ and $\beta = 2$) within a cylindrical domain using boundary measurements. This analysis was later extended to fractional models with space-time dependent diffusion coefficients in \cite{JinKianZhou:2021} through a novel perturbation argument. The authors also developed a numerical algorithm for efficiently and accurately reconstructing the source component. 
Niu, Helin, and Zhang \cite{niu-helin-zhang-2020} applied a regularization technique to reconstruct random sources in stochastic fractional diffusion equations. Meanwhile, Jiang, Li, Liu, and Yamamoto \cite{jiang-li-liu-yamamoto-2017} reformulated the inverse source problem as an optimization problem, proposing an iterative thresholding method for its numerical solution. More recently, Janno and Kian \cite{janno-kian-2023} investigated inverse problems involving the determination of source terms that vary spatially or temporally, utilizing a posteriori boundary measurements.

The aforementioned studies are primarily concerned with the case where the fractional order $\alpha$ lies within the range $(0,1]$. However, theoretical advancements in identifying spatially varying sources for $\alpha \in (1,2)$ remain relatively scarce. Specifically, the uniqueness of identifying spatially dependent sources was established by applying the analytic nature of solutions and a novel unique continuation principle, under the assumption that all coefficients are time-independent, as detailed in \cite{Cheng2021uniqueness}. In another study, Liao and Wei \cite{liao-wei-2019} employed the Levenberg--Marquardt method to simultaneously reconstruct the fractional order and the spatial source term. Yan and Wei \cite{Yan2020determine} investigated the determination of a spatially dependent source term in a time-fractional diffusion-wave equation using partial noisy boundary data, establishing uniqueness via the Titchmarsh convolution theorem and the Duhamel principle. Liu, Hu and Yamamoto \cite{Liu2021inverse} studied the identification of moving source profiles in time-fractional diffusion-wave equations, confirming the unique determination of various source profiles by applying the unique continuation principle. Lassas, Li, and Zhang \cite{lassas-li-zhang-2023} explored the unique continuation principle for stochastic time-fractional diffusion-wave equations and its application to inverse source problems. Yamamoto \cite{Yamamoto2023uniqueness} examined the uniqueness of the inverse source problem for spatially varying factors by analyzing the decay of data as time approaches infinity, assuming that the source remains inactive during the observation period. The identification of additional parameters, such as fractional orders, initial conditions, and coefficients, is a rapidly growing area of research. For a comprehensive review on this topic, we refer readers to \cite{Jin2015turorial, LiLiuYamamoto2019inverse, LiYamamoto2019inverse, LiuLiYamamoto2019inverse}.

To the authors' knowledge, almost all studies on the inverse source problem for fractional equations have focused on establishing uniqueness results, primarily due to the lack of tools such as Carleman estimates. By assuming a temporally and spatially separable source term, we investigate the conditional stability of determining the source term based on one-end boundary observations. This represents a key innovative aspect of our article.


For giving our stability result, we introduce some notations and definitions. Fixing $\gamma\ge0$ and any $M>0$, we define an admissible set of unknown source component $f$ by 
$$
\mathcal U_{M,\gamma} := \left\{ \sum_{|n|=1}^\infty f_n X_n; |f_n| n^\gamma \le M,\ \forall n=1,2,\cdots  \right\}.
$$
Here $X_n$ is the eigenfunction of the operator $D_x^\beta$ and will be given later in the next section. 

Our main theorems in this article are as follows. 
\begin{thm}[Uniquely recovery] \label{thm-uniqueness}
Assuming $\alpha,\beta\in(1,2)$, $T>0$ and $u\in C([0,T]\times(0,1])$ satisfies the problem \eqref{eq-gov} with $f \in \mathcal U_{M,0}$ and $\lambda\in C^1[0,T]$. Then the source term $f$ can be uniquely determined by the observation data $\phi(t)$, that is, $f$ must be identically zero if $\phi=0$ in $(0,T)$.
\end{thm}
Moreover, assuming $\alpha=\beta$, we can obtain a conditional stability result for our inverse source problem.
\begin{thm}[Conditional stability] \label{thm-stability}
Assuming $\alpha=\beta\in(1,2)$, $T\ge1$ and $u\in C([0,T]\times(0,1])$ satisfies the problem \eqref{eq-gov} with $f \in C[0,1] \cap \mathcal U_{M,\gamma}$,  $\gamma>2-\alpha$ and $\lambda\in C^1[0,T]$. Then the source term $f$ can be estimated by
$$
\|f\|_{C[0,1]} \le C\|\phi\|_{C[0,1]}^\theta,
$$
where $\theta\in (0,1)$ and $C>0$ are constants depending only on $\alpha,\gamma,M,T,\|\lambda\|_{C^1[0,T]}$.
\end{thm}

The outline of this article is as follows: In Section \ref{sec-bi}, we present auxiliary results, focusing primarily on the bi-orthogonal system associated with the fractional operator $\partial_x^\alpha$. These results, combined with the refined analysis of the direct problem established in Section \ref{sec-forward}, lead to the proof of the main theorems, which is provided in Section \ref{sec-proof}. A numerical scheme for the proposed source term determination is introduced in Section \ref{sec-num}. Finally, concluding remarks are presented in Section \ref{sec-con}.

\section{Bi-orthogonal system}\label{sec-bi}
In this section, we will first present several key properties related to the zeros of the Mittag--Leffler function $E_{\alpha,\beta}(z)$, which is defined as:
\begin{equation}\label{def-ml}
E_{\alpha,\beta}(z):= \sum_{k=0}^\infty \frac{z^k}{\Gamma(\alpha k+\beta)},\quad z\in \mathbb{C}.
\end{equation}
Subsequently, we will construct a bi-orthogonal system derived from the associated spectral problem and its adjoint problem, which will play a crucial role in the forthcoming discussions.

To achieve this, using the method of separation of variables, we introduce the spectral problem associated with the governing equation \eqref{eq-gov}:
\begin{equation}
\label{eq-eigen}
\left\{
\begin{aligned}
& D_x^\beta X(x) = \lambda X(x), \quad x \in (0,1),\\
& X(0) = X(1) = 0.
\end{aligned}
\right.
\end{equation}
The spectral problem has been studied extensively in many existing works; see, e.g., \cite{Aleroev2013BoundaryvaluePF, JinLazarovPasciakZhou:2014, Jin2015turorial}. The eigenfunctions of the spectral problem are given by:
\begin{equation}\label{def-Xn}
    X_n(x) = \Big\{ x^{\beta-1} E_{\beta,\beta}(\lambda_n x^\beta) \Big\}_{|n|=1}^\infty,
\end{equation}
where $\lambda_n\in\mathbb C$ denotes the eigenvalues associated with the spectral problem. These eigenvalues $\lambda_n$ are the zeros of the Mittag--Leffler function $E_{\beta,\beta}(\lambda)$, satisfying the property $\overline{\lambda_n} = \lambda_{-n}$.

Our analysis critically depends on the asymptotic behavior and distribution of the eigenvalues $\lambda_n$, as described in the following lemma. For detailed proofs, we refer the reader to \cite[Theorems 2.1.1 and 4.2.1]{2013Distribution}.

\begin{lem}
\label{lem-eigen}
Letting $\beta\in (1,2)$, then the eigenvalues $\lambda_n$, that are the zeros of the function $E_{\beta,\beta}(z)$, satisfy the following relations
\begin{enumerate}
\item $|\arg \lambda_n|>\frac{\beta\pi}2$, $n=\pm1,\pm2,\cdots$;
\item $0<|\lambda_{|n|}| < |\lambda_{|n|+1}|$, $n=\pm1,\pm2,\cdots$;
\item $|\lambda_n|\sim O(|n|^\beta)$  as $n$ tends to infinity.
\end{enumerate}
\end{lem}

In the case of $\beta = 2$, with the help of Dirichlet eigensystem $\{\lambda_n,X_n(\cdot)\}_{|n|=1}^\infty$ of the problem \eqref{eq-eigen}, Sakamoto and Yamamoto \cite{Sakamoto+Yamamoto-2011} proved the unique existence of the solution to \eqref{eq-gov}, gave several regularity estimates, and analyzed some related inverse problems. However, as pointed out in Section 3 of \cite{Aleroev2013BoundaryvaluePF}, the set $\{X_n(\cdot)\}_{|n|=1}^\infty$ is complete in $L^2(0,1)$ but not orthogonal, which makes the approaches based on eigenfunction expansion in \cite{Sakamoto+Yamamoto-2011} cannot work anymore. This is one of the main difficulties in our paper.

For this, we consider the adjoint problem of the spectral problem \eqref{eq-eigen},
\begin{equation}
\label{adjoint-eigen}
\left\{
\begin{aligned}
& {^BD_x^\beta} X(x) = \lambda X(x), \quad x\in (0,1),\\
&X(0)=X(1)=0,
\end{aligned}
\right.
\end{equation}
where ${^BD_x^\beta}$ denotes the backward Riemann-Liouville derivative defined by
$$
{^BD_x^\beta} \varphi(x) = \frac1{\Gamma(2-\beta)} \frac{d}{dx}\int_x^1  (\xi - x)^{1-\beta} \varphi'(\xi) d\xi.
$$
The adjoint problem \eqref{adjoint-eigen} has eigenfunctions $Y_n(x)$ corresponding to the conjugate $\overline\lambda_n$ as that of the problem \eqref{eq-eigen}, where 
\begin{equation}\label{def-Yn}
Y_n(x) = \overline{X_n(1-x)} = (1-x)^{\beta-1} E_{\beta,\beta}(\overline{\lambda_n} (1-x)^\beta),\quad n=\pm1,\pm2,\cdots.
\end{equation}
The sets $\{X_n(\cdot)\}_{|n|=1}^\infty$ and $\{Y_n(\cdot)\}_{|n|=1}^\infty$ form a bi-orthogonal system of functions, that is, 
$$
\langle X_n, Y_m\rangle_{L_x^2(0,1)} := \int_0^1 X_n(x) \overline{Y_m(x)} dx = 0,\quad m\ne n,
$$
which can be proved in view of integration by parts, see also e.g., \cite{Aleroev2013BoundaryvaluePF}. Moreover, we have the following lemmata related to the bi-orthogonal systems $\{X_n\}_{|n|=1}^\infty$ and $\{Y_n\}_{|n|=1}^\infty$. First, we have that $\langle X_n, Y_n \rangle_{L_x^2(0,1)}$ is not vanished for any $n=\pm1,\pm2,\cdots$. 
\begin{lem}\label{lem-XY}
Letting $X_n$ and $Y_n$ be defined by \eqref{def-Xn} and \eqref{def-Yn} separately. The $L^2(0,1)$-inner product $\langle X_n,Y_n \rangle_{L_x^2(0,1)} $ admits
$$
\langle X_n,Y_n \rangle_{L_x^2(0,1)} \neq 0,\quad \forall n=\pm1,\pm2,\cdots.
$$
\end{lem}
\begin{proof}
If not, there exists an integer $n_0$ such that $\langle X_{n_0},Y_{n_0} \rangle_{L_x^2(0,1)} = 0$. On the other hand, since the sets $\{X_n\}_{|n|=1}^\infty$ and $\{Y_n\}_{|n|=1}^\infty$ form a bi-orthogonal system of functions, we see that $\langle X_n,Y_{n_0}\rangle_{L_x^2(0,1)}=0$ for any $n=\pm1,\pm2,\cdots$. We must have $Y_{n_0}=0$ in $(0,1)$ in view of the completeness of the set $\{X_n\}_{|n|=1}^\infty$. But it is obvious that $Y_n$ is not identically zero in $(0,1)$. We finish the proof of the lemma. 
\end{proof}
Moreover, the following lemma gives a representation of $\langle X_n,Y_n \rangle_{L_x^2(0,1)}$.
\begin{lem}\label{lem-XY'}
Under the same settings and assumptions in Lemma \ref{lem-XY}, then the $L^2(0,1)$-inner product $\langle X_n,Y_n \rangle_{L_x^2(0,1)} $ can be rephrased as follows 
$$
\langle X_n,Y_n \rangle_{L_x^2(0,1)} = \sum_{k=1}^\infty \frac{k\lambda_n^{k-1}}{\Gamma(\beta k+\beta)}
$$
for any $n=\pm1,\pm2,\cdots$. Here, $\lambda_n\in\mathbb C$ are the zeros of the Mittag--Leffler function $E_{\beta,\beta}(\cdot)$.
\end{lem}
\begin{proof}
From the definition of the Mittag--Leffler functions, by a direct calculation, we can obtain 
\begin{equation}
\label{eq-XY}
\begin{aligned}
\langle X_n,Y_n \rangle_{L_x^2(0,1)} 
=& \int_0^1 x^{\beta-1} E_{\beta,\beta}(-\lambda_n x^\beta) (1-x)^{\beta-1} E_{\beta,\beta}(-\lambda_n (1-x)^\beta) dx
\\
=& \int_0^1 \sum_{i=0}^\infty \frac{\lambda_n^i x^{\beta(i+1)-1}}{\Gamma(\beta i+\beta)} \sum_{j=0}^\infty \frac{\lambda_n^j (1-x)^{\beta(j+1)-1}}{\Gamma(\beta j+\beta)}dx.
\end{aligned}
\end{equation}
There is a product of two infinite series in the above integration, which can be dealt with via the Cauchy type product of series as follows
$$
\sum_{i=0}^\infty \frac{\lambda_n^i x^{\beta(i+1)-1}}{\Gamma(\beta i+\beta)} \sum_{j=0}^\infty \frac{\lambda_n^j (1-x)^{\beta(j+1)-1}}{\Gamma(\beta j+\beta)}
=\sum_{k=0}^\infty \lambda_n^k \sum_{i+j=k} \frac{x^{\beta(i+1)-1} (1-x)^{\beta(j+1)-1} }{ \Gamma(\beta i+\beta) \Gamma(\beta j+\beta)}.
$$
Substitute the above formula into \eqref{eq-XY}, and by the Fubini lemma, we further see that
$$
\langle X_n,Y_n \rangle_{L_x^2(0,1)} = 
\sum_{k=0}^\infty \lambda_n^k \sum_{i+j=k}  \frac{\int_0^1 x^{\beta(i+1)-1} (1-x)^{\beta(j+1)-1} dx}{ \Gamma(\beta i+\beta) \Gamma(\beta j+\beta)},
$$
which combined with the fact that
$$
\int_0^1 x^{\beta(i+1)-1} (1-x)^{\beta(j+1)-1} dx
=\frac{\Gamma(\beta i+\beta)\Gamma(\beta j + \beta)}{\Gamma(\beta(i+j+2))}
$$
implies that
$$
\langle X_n,Y_n \rangle_{L_x^2(0,1)} =
\sum_{k=1}^\infty \frac{k\lambda_n^{k-1}}{\Gamma(\beta k+\beta)},\quad \forall n=\pm1,\pm2,\cdots.
$$
This completes the proof of the lemma.
\end{proof}

Based on the formula in Lemma \ref{lem-XY}, we can further give a lower bound for $|\langle X_n,Y_n\rangle_{L_x^2(0,1)}|$. We have
\begin{lem}\label{lem-XnYn}
Assume $\lambda_n\in\mathbb C$ is the eigenvalue of the spectral problem \eqref{eq-eigen} such that $E_{\alpha,\alpha}(\lambda_n)=0$, $n=\pm1,\pm2,\cdots$. Then $|\langle X_n,Y_n\rangle_{L_x^2(0,1)}|$ admits a lower bound estimate:
\begin{equation}
\label{esti-XY}
|\langle X_n,Y_n\rangle_{L_x^2(0,1)}| \ge \frac{C}{|\lambda_n|^3},\quad \forall n=\pm1,\pm2,\cdots,
\end{equation}
where the constant $C$ is positive and is independent of $n$, but may depends on $\beta$.
\end{lem}
\begin{proof}
We set 
\begin{equation}
\label{def-g}
g(x) := \sum_{k=1}^\infty \frac{kx^{k-1}}{\Gamma(\beta k+\beta)},
\end{equation}
then by a direct calculation, it is not difficult to obtain that
$$
\frac{d}{dx} E_{\beta,\beta}(x) = \sum_{k=0}^\infty \frac{d}{dx} \frac{x^k}{\Gamma(\beta k+\beta)}= g(x).
$$
Now, in view of the following formulas
$$
\frac{d}{dx} (x^{\beta-1} E_{\beta,\beta}(x^\beta)) 
= (\beta-1)x^{\beta-2} E_{\beta,\beta}(x^\beta) + \beta x^{2\beta-2} g(x^\beta)
$$
and
$$
\frac{d}{dx} (x^{\beta-1} E_{\beta,\beta}(x^\beta))
=x^{\beta-2} E_{\beta,\beta-1}(x^\beta),
$$
we see that
$$
g(x^\beta) = \frac{ E_{\beta,\beta-1}(x^\beta) - (\beta-1)E_{\beta,\beta}(x^\beta) }{\beta x^\beta},
$$
that is
$$
g(x) = \frac{ E_{\beta,\beta-1}(x) - (\beta-1)E_{\beta,\beta}(x) }{\beta x}.
$$
Now from the asymptotic expansion of the Mittag--Leffler functions
\begin{equation}\label{eq-asymp}
E_{\alpha,\beta}(z) = -\sum_{k=1}^N \frac{z^{-k}}{\Gamma(\beta-\alpha k)} + O(|z|^{-1-N}),\quad |z|\to\infty,\quad \frac{\pi\alpha}2 \le |\arg z|\le \pi,
\end{equation}
see e.g., Theorem 1.4 on p.33 in Podlubny \cite{Podlubny-1999}, and noting that $\Gamma(0)=\Gamma(-1)=\infty$, we obtain
\begin{equation*}
\begin{aligned}
E_{\beta,\beta-1}(x) &\sim -\frac{x^{-2}}{\Gamma(-1-\beta)},
\\
E_{\beta,\beta}(x) &\sim -\frac{x^{-2}}{\Gamma(-\beta)},
\end{aligned}
\quad |x|\to\infty,\quad \frac{\pi\beta}2 \le |\arg x|\le \pi,
\end{equation*}
which combined with \eqref{def-g} implies 
\begin{align*}
g(x) 
&\sim \frac1{\beta x} \left[ -\frac{x^{-2}}{\Gamma(-1-\beta)} -\frac{x^{-2}}{\Gamma(-\beta)} \right]
\\
&= \frac2{\beta+1} \frac{-1}{\Gamma(-1-\beta)} \frac1{x^3}
= \frac2{\Gamma(-\beta)} \frac1{x^3},\quad |x|\to\infty,\quad \frac{\pi\beta}2\le|\arg x|\le\pi.
\end{align*}
Here, in the last step, we used the identity of the Gamma function: $\Gamma(-1-\beta)(-1-\beta) = \Gamma(-\beta)$. 

Finally, we get the asymptotic estimate for $g(\lambda_n)$:
$$
g(\lambda_n) \sim \frac2{\Gamma(-\beta)} \frac1{|\lambda_n|^3},\quad n\to\infty, \quad\frac{\pi\beta}2 \le |\arg \lambda_n|\le \pi,
$$
which combined with the formula in Lemma \ref{lem-XY}, we derive that
$$
|\langle X_n,Y_n\rangle_{L_x^2(0,1)}| \ge \frac{C}{|\lambda_n|^3},\quad n=\pm1,\pm2,\cdots.
$$
We complete the proof of the lemma.
\end{proof}
On the basis of the above lemmas, we have the following useful corollary.
\begin{cor}\label{cor-simple}
 Letting $\beta\in(1,2)$,  and $\lambda\in\mathbb C$ satisfy $E_{\beta,\beta}(\lambda)=0$, then $E_{\beta,\beta-1}(\lambda)\neq0$.
\end{cor}

\section{Forward problem}\label{sec-forward}

Noting that $\{X_n\}_{|n|=1}^\infty$ forms a complete basis of $L^2(0,1)$, we assume that the unknown source profile $f$ is in $C[0,1]$ and can be written in the form of $f:=\sum_{|n|=1}^\infty f_n X_n$. Moreover, we can see that 
$$
f_n=\frac{\langle f,Y_n \rangle_{L_x^2(0,1)}}{\langle X_n,Y_n \rangle_{L_x^2(0,1)}}$$ by using the fact that $\{X_n\}_{|n|=1}^\infty$ and $\{Y_n\}_{|n|=1}^\infty$ form a bi-orthogonal system of functions. From the eigenfunction expansion argument, we see that the solution $u$ to the problem \eqref{eq-gov} can be represented as follows
\begin{equation}\label{sol-u}
u(x,t) = \sum_{|n|=1}^\infty \int_0^t \lambda(t-\tau) \tau^{\alpha-1} E_{\alpha,\alpha}(\lambda_n\tau^\alpha) d\tau \frac{\langle f,Y_n \rangle_{L_x^2(0,1)}}{\langle X_n,Y_n \rangle_{L_x^2(0,1)}} X_n(x).
\end{equation}
In fact, we have the wellposedness for the problem \eqref{eq-gov}.
\begin{thm}\label{thm-fp}
The solution $u$ to the problem \eqref{eq-gov} admits the representation
$$
u(x,t) = \sum_{|n|=1}^\infty \int_0^t \lambda(t-\tau) \tau^{\alpha-1} E_{\alpha,\alpha}(\lambda_n\tau^\alpha)  d\tau \frac{\langle f,Y_n \rangle_{L_x^2(0,1)}}{\langle X_n,Y_n \rangle_{L_x^2(0,1)}} X_n(x),
$$
 and for $\varepsilon\in(0,\alpha-1]$ be any fixed, $u$ can be estimated by
\begin{equation}\label{esti-u}
 \|x^{2+\varepsilon} u(x,t)\|_{C([0,1]\times[0,T])} \le C\|f\|_{L^\infty(0,1)}.
\end{equation}
Here, the constant $C>0$ depends only on $\varepsilon$, $T$, $\alpha$, $\beta$ and $\|\lambda\|_{L^\infty(0,T)}$.
Moreover, if $f\in \mathcal U_{M,0}$, we have $u\in C([0,1]\times[0,T])$ and $u_x \in C((0,1]\times[0,T])$.
\end{thm}
\begin{proof}
In order to get the estimate \eqref{esti-u}, it suffices to evaluate each term in the equation \eqref{sol-u}. Firstly, from the properties of the Mittag--Leffler functions, $x^{2+\varepsilon} X_n(x)$ can be estimated as follows
\begin{align*}
|x^{2+\varepsilon} X_n(x)| 
\le \frac{Cx^{\beta+1+\varepsilon}}{(1+|\lambda_n|x^\beta)^2}
\le  \frac{C(|\lambda_n| x^\beta)^{1+\frac{1+\varepsilon}\beta}}{(1+|\lambda_n|x^\beta)^2} |\lambda_n|^{-1-\frac{1+\varepsilon}\beta}, \quad x\in(0,1).
\end{align*}
Moreover, in view of the fact that 
\begin{equation}\label{esti1}
\sup_{x\in[0,1]}\frac{x^\gamma}{(1+x)^2} \le C, \text{ where $\gamma\in[0,2]$, }
\end{equation}
we further see that
\begin{align*}
|x^{2+\varepsilon} X_n(x)| 
\le C|\lambda_n|^{-1-\frac{1+\varepsilon}\beta}, \quad x\in(0,1).
\end{align*}
Now we turn to estimating $\langle f,Y_n \rangle_{L_x^2(0,1)}$. Again by the use of the asymptotic properties of the Mittag--Leffler functions in \eqref{eq-asymp}, noting $f\in L^\infty(0,1)$,  a direct calculation yields
\begin{align*}
|\langle f,Y_n \rangle_{L_x^2(0,1)}|
\le& C\|f\|_{L^\infty(0,1)} \int_0^1 \frac{x^{\beta-1}}{(1+|\lambda_n|x^\beta)^2} dx\\
\le& C\|f\|_{L^\infty(0,1)} \int_0^1 \frac{(|\lambda_n|x^\beta)^{1-\frac{\varepsilon}{4\beta}}}{(1+|\lambda_n|x^\beta)^2} x^{\frac\varepsilon4-1} |\lambda_n|^{\frac\varepsilon{4\beta}-1}dx
\\
\le& C\|f\|_{L^\infty(0,1)} |\lambda_n|^{\frac\varepsilon{4\beta}-1}\int_0^1  x^{\frac\varepsilon4-1} dx .
\end{align*}
Here in the last inequality, we again used the estimate \eqref{esti1}. Finally, we get 
\begin{align*}
|\langle f,Y_n \rangle_{L_x^2(0,1)}| \le \frac{4C}{\varepsilon}\|f\|_{L^\infty(0,1)} |\lambda_n|^{\frac\varepsilon{4\beta}-1}.
\end{align*}
By a similar argument, we can arrive at the following inequalities  
$$
\begin{aligned}
&\left|\int_0^t \lambda(t-\tau) \tau^{\alpha-1} E_{\alpha,\alpha}(\lambda_n\tau^\alpha) d\tau \right| \\
\le& \|\lambda\|_{L^\infty(0,T)} \int_0^T |t^{\alpha-1} E_{\alpha,\alpha}(\lambda_n t^\alpha)| dt
\le \frac{4CT^{\frac{\varepsilon}{4}}}{\varepsilon} \|\lambda\|_{L^\infty(0,T)} |\lambda_n|^{\frac\varepsilon{4\alpha}-1}.
\end{aligned}
$$
Combining all the above estimates, we get
\begin{align*}
|x^{2+\varepsilon}u(x,t)| \le C\|\lambda\|_{L^\infty(0,T)}\|f\|_{L^\infty(0,1)} \sum_{|n|=1}^\infty  |\lambda_n|^{3+\frac\varepsilon{4\alpha}-1+\frac\varepsilon{4\beta}-1-1-\frac{1+\varepsilon}{\beta}},
\end{align*}
Moreover, noting the asymptotic estimate $|\lambda_n|\sim O(|n|^\beta)$, as $n\to\infty$ in Lemma  \ref{lem-eigen}, we see that 
$$
|\lambda_n|^{3+\frac\varepsilon{4\alpha}-1+\frac\varepsilon{4\beta}-1-1-\frac{1+\varepsilon}{\beta}}
= |\lambda_n|^{-\frac{1}{\beta} +\varepsilon(\frac1{4\alpha}+\frac1{4\beta}- \frac{1}{\beta})  }
\sim n^{-1+\varepsilon(\frac\beta{4\alpha}-\frac34)  },
$$
by noting that $\beta/\alpha <2$, we further see that
$$
|\lambda_n|^{3+\frac\varepsilon{4\alpha}-1+\frac\varepsilon{4\beta}-1-1-\frac{1+\varepsilon}{\beta}}
 \le Cn^{-1-\frac{\varepsilon}4},
$$
which implies that the infinite series 
$$
\sum_{|n|=1}^\infty  |\lambda_n|^{3+\frac\varepsilon{4\beta}-1+\frac\varepsilon{4\beta}-1-1-\frac{1+\varepsilon}{\beta}}<\infty,
$$
and therefore we see that
\begin{align*}
|x^{2+\varepsilon}u(x,t)| \le C\|\lambda\|_{L^\infty(0,T)}\|f\|_{L^\infty(0,1)},\quad (x,t)\in(0,1)\times(0,T).
\end{align*}
We finish the proof of the first part of the theorem. Next we assume the function $f\in\mathcal U_{M,0}$ and then we write $f$ by 
$$
f=\sum_{|n|=1}^\infty f_n X_n,\quad |f_n| \le M.
$$
Then by arguments similar to the first step, a direct calculation derives
\begin{align*}
|u(x,t)| \le& C\|\lambda\|_{L^\infty(0,T)} \sum_{|n|=1}^\infty \int_0^t \tau^{\alpha-1}|E_{\alpha,\alpha}(\lambda_n\tau^\alpha)| d\tau  \left| f_n\right|\, \left|X_n(x)\right|
\\
\le& CM \|\lambda\|_{L^\infty(0,T)} \sum_{|n|=1}^\infty \left(\frac{T^{\frac{\varepsilon}{4}}}{\varepsilon} |\lambda_n|^{\frac\varepsilon{4\alpha} -1} \right) \frac{Cx^{\beta-1}}{(1+|\lambda_n|x^\beta)^2}.
\end{align*}
Here in the second inequality we used the property of the Mittag--Leffler function and the argument for estimating $\int_0^t \tau^{\alpha-1} |E_{\alpha,\alpha}(\lambda_n\tau^\alpha)| d\tau$ in the first step. Furthermore, using the inequality \eqref{esti1}, we can obtain
\begin{equation*}
\begin{split}
|u(x,t)| \le& C \|\lambda\|_{L^\infty(0,T)} \sum_{|n|=1}^\infty  |\lambda_n|^{\frac\varepsilon{4\alpha} -1}  \frac{C(|\lambda_n| x^\beta)^{1-\frac1{\beta}}}{(1+|\lambda_n|x^\beta)^2} |\lambda_n|^{\frac1\beta-1}
\\
\le& C\sum_{|n|=1}^\infty  |\lambda_n|^{\frac\varepsilon{4\alpha} -1}  |\lambda_n|^{\frac1\beta-1},
\end{split}
\end{equation*}
which combined with the fact $|\lambda_n|\sim |n|^\beta$, as $n\to\infty$,  implies
\begin{align*}
|u(x,t)| \le C\sum_{|n|=1}^\infty |n|^{1+\frac{\varepsilon\beta}{4\alpha} -2\beta}.
\end{align*}
Now by noting $\beta>1$, we can choose $\varepsilon>0$ being sufficiently small so that $1+\frac{\varepsilon\beta}{4\alpha} -2\beta <-1$, therefore we finally get $|u(x,t)| \le C$.

For $f\in\mathcal U_{M,0}$, we next prove that $u_x(x,t)\in C((0,1]\times[0,T])$. For this, we set 
$$
u_N(x,t) := \sum_{|n|=1}^N \int_0^t \lambda(t-\tau) \tau^{\alpha-1} E_{\alpha,\alpha}(\lambda_n\tau^\alpha) d\tau f_n X_n(x).
$$
It is not difficult to check that $u_N(x,t)$ is differentiable with respect to $x$, and we have
$$
\frac{\partial}{\partial x}u_N(x,t) = \sum_{|n|=1}^N \int_0^t \lambda(t-\tau) \tau^{\alpha-1} E_{\alpha,\alpha}(\lambda_n\tau^\alpha) d\tau f_n x^{\beta-2} E_{\beta,\beta-1}(\lambda_nx^\beta).
$$
We can follow the above treatment to analoguely obtain
\begin{align*}
\left| \frac{\partial}{\partial x}u_N(x,t) \right| 
\le& \sum_{|n|=1}^N |f_n|\, \frac{4CT^{\frac{\varepsilon}{4}}}{\varepsilon} \|\lambda\|_{L^\infty(0,T)} |\lambda_n|^{\frac\varepsilon{4\alpha}-1}  \frac{Cx^{\beta-2}}{1+|\lambda_n|x^\beta}
\\
\le&  C x^{-2} \|\lambda\|_{L^\infty(0,T)}\sum_{|n|=1}^N  |\lambda_n|^{\frac\varepsilon{4\alpha}-1} |\lambda_n|^{-1}.
\end{align*}
By virtue of the asymptotic estimate of $\lambda_n$ in Lemma \ref{lem-eigen}, we finally arrive at
\begin{align*}
\left| \frac{\partial}{\partial x}u_N(x,t) \right| 
\le  C x^{-2} \|\lambda\|_{L^\infty(0,T)}\sum_{|n|=1}^N |n|^{\frac{\varepsilon\beta}{4\alpha}-2\beta}.
\end{align*}
Now by choosing $\varepsilon>0$ sufficiently small, we have $\frac{\varepsilon\beta}{4\alpha}-2\beta <-1$. Consequently, for any $\delta>0$, we can see that $\lim_{N\to\infty} \frac{\partial}{\partial x} u_N(x,t)$ uniformly for any $(x,t)\in [\delta,1]\times[0,T]$. We finish the proof of the theorem.
\end{proof}

\section{Proof of the main theorems}\label{sec-proof}
In the previous section, we show that the solution $u_x \in C((0,1]\times[0,T])$, which makes the observation $u_x(1,t)$ well defined for any $t\in(0,T)$. We will first show the proof of the first main result, that is, that $u_x(1,t)=0$ implies $f=0$.
\begin{proof}[Proof of Theorem \ref{thm-uniqueness}]
Assume $f\in \mathcal U_{M,\gamma}$ with $\gamma=0$, and from the representation formula of the solution in \eqref{sol-u}, we see that
$$
\frac{\partial}{\partial x}u(x,t) = \sum_{|n|=1}^\infty \int_0^t \lambda(t-\tau) \tau^{\alpha-1} E_{\alpha,\alpha}(\lambda_n\tau^\alpha) d\tau f_n x^{\beta-2} E_{\beta,\beta-1}(\lambda_nx^\beta),\quad x\in(0,1].
$$
From the observation assumption 
$u_x(1,t)=0$ for $t\in[0,T]$, letting $x=1$ in the above equation further implies
$$
0=\frac{\partial}{\partial x}u(1,t) = \int_0^t \lambda(t-\tau) \sum_{|n|=1}^\infty  \tau^{\alpha-1} E_{\alpha,\alpha}(\lambda_n\tau^\alpha)  f_n E_{\beta,\beta-1}(\lambda_n) d\tau,\quad t\in[0,T].
$$
By an argument similar to the proof of Theorem \ref{thm-fp}, it is not difficult to check that 
$$
\sum_{|n|=1}^N  E_{\alpha,\alpha}(\lambda_n z^\alpha)  f_n E_{\beta,\beta-1}(\lambda_n)
$$ is uniformly convergent as $N\to\infty$ and $f\in\mathcal U_{M,0}$, which implies from the Weierstrass theorem that 
$$
\sum_{|n|=1}^\infty  E_{\alpha,1}(\lambda_n z^\alpha)  f_n E_{\beta,\beta-1}(\lambda_n)
$$
is analytic with respect to $\Re z>0$. In view of the Titchmarsh theorem for convolution, we conclude that 
$$
\sum_{|n|=1}^\infty  E_{\alpha,1}(\lambda_n t^\alpha)  f_n E_{\beta,\beta-1}(\lambda_n) = 0,\quad t>0.
$$
Now taking Laplace transforms on both sides of the above equation, we see that
$$
\sum_{|n|=1}^\infty \frac{s^{\alpha-1}}{s^\alpha - \lambda_n}  f_n E_{\beta,\beta-1}(\lambda_n) = 0,\quad s>0.
$$
By letting $\eta=s^\alpha$, we see that 
$$
\sum_{|n|=1}^\infty \frac{1}{\eta - \lambda_n}  f_n E_{\beta,\beta-1}(\lambda_n) = 0,\quad \eta\in\mathbb C\setminus\{\lambda_n\}_{n=1}^\infty,
$$
which implies
$$
f_n E_{\beta,\beta-1}(\lambda_n) =0,\quad n=1,2,\cdots.
$$
Moreover, since $E_{\beta,\beta-1}(\lambda_n)\neq0$ from Corollary \ref{cor-simple}, we see that $f_n=0$ for any $n=1,2,\cdots$. We finish the proof of the theorem.
\end{proof}

In next step, we let $T\ge1$ and $\alpha=\beta$, we give a proof of the second main result. For this, we will firstly construct a bi-orthogonal system to the sequence $\{KX_n(t)\}_{|n|=1}^\infty$, where the operator $K$ is defined by 
\begin{equation}\label{def-K}
K\varphi(t) := \int_0^t \lambda(t-\tau) \varphi(\tau) d\tau,\quad \forall \varphi\in L^2(0,T).
\end{equation}
It is not difficult to check that $K\in \mathcal L(L^2(0,T), {_0 H^1(0,T)})$, that is, $K$ is a linear bounded operator from $L^2(0,T)$ to ${_0 H^1(0,T)}:=\{q\in H^1(0,T); q(0)=0\}$. Moreover, Yamamoto \cite{Yamamoto-1993} asserts that there exists an inverse operator $K^{-1}\in \mathcal L({_0 H^1(0,T)}, L^2(0,T))$. Then we define a system $Z_n$ in ${_0H^1(0,T)}$ by 
$$
Z_n(t):= (K^{-1})^* Y_n(t),\quad t\in(0,T).
$$
\begin{lem}
For any $m,n=\pm1,\pm2,\cdots$, we have
$$
\langle Z_n, KX_m \rangle_{L_t^2(0,T)} = \langle X_m, Y_n \rangle_{L_t^2(0,T)}.
$$
\end{lem}
\begin{proof}
From the definition of $Z_n$ and $KX_n$, it follows that
$$
\langle Z_n, KX_n \rangle_{L_t^2(0,T)} = \langle (K^{-1})^* Y_n, KX_m \rangle_{L_t^2(0,T)} = \langle Y_n, K^{-1} K X_m \rangle_{L_t^2(0,T)} = \langle Y_n, X_m \rangle_{L_t^2(0,T)}.
$$
This completes the proof of the lemma.
\end{proof}
Since the observation time interval $[0,T]$ is greater than $[0,1]$, we can use less information over the time length, that is, without loss of generality, we can  assume $T=1$ and then the above lemma asserts that the system $\{Z_n\}_{|n|=1}^\infty$ is the desired bi-orthogonal system to the set $\{KX_n(t)\}_{|n|=1}^\infty$ under the inner product $L^2_t(0,1)$, that is,
$$
\langle Z_n, KX_m \rangle_{L_t^2(0,1)} = 0,\quad \forall m\ne n.
$$
Now we are ready to give the proof of the conditional stability of the inverse problem.
\begin{proof}[Proof of Theorem \ref{thm-stability}]
Assume $f\in \mathcal U_{M,\gamma}$ with $\gamma>2-\alpha$, and from the representation formula of the solution in \eqref{sol-u}, we see that
$$
\frac{\partial}{\partial x}u(x,t) = \sum_{|n|=1}^\infty \int_0^t \lambda(t-\tau) X_n(\tau) d\tau f_n x^{\alpha-2} E_{\alpha,\alpha-1}(\lambda_nx^\alpha).
$$
Letting $x=1$ in the above equation further implies
$$
\frac{\partial}{\partial x}u(1,t) = \sum_{|n|=1}^\infty \int_0^t \lambda(t-\tau) X_n(\tau) d\tau f_n E_{\alpha,\alpha-1}(\lambda_n).
$$
Now multiplying both sides by $\{Z_n\}$ and integrating on $t\in[0,1]$, we can arrive at the following equality
$$
\langle \frac{\partial}{\partial x}u(1,t), Z_n \rangle_{L_t^2(0,1)} = \langle f,Y_n \rangle_{L_x^2(0,1)} E_{\alpha,\alpha-1}(\lambda_n),
$$
which implies
$$
f_n= \frac{\langle f,Y_n \rangle_{L_x^2(0,1)}}{\langle X_n,Y_n \rangle_{L_x^2(0,1)}} = \frac{\langle \frac{\partial}{\partial x}u(1,t), Z_n \rangle_{L_t^2(0,1)}}{ \langle X_n,Y_n \rangle_{L_x^2(0,1)} E_{\alpha,\alpha-1}(\lambda_n)}.
$$
Based on this representation of $f$ via the additional data $u_x(1,t)$, we can estimate $f$ by 
\begin{align*}
|f(x)| =& \left|\sum_{|n|=1}^\infty f_n X_n(x) \right| \le \sum_{|n|=1}^{N-1} |f_n| |X_n(x)| + \sum_{|n|=N}^\infty |f_n| | X_n(x)|
\\
=&  \sum_{|n|=1}^{N-1}  \left| \frac{\langle \frac{\partial}{\partial x}u(1,t), Z_n \rangle_{L_t^2(0,1)}}{ \langle X_n,Y_n \rangle_{L_x^2(0,1)} E_{\alpha,\alpha-1}(\lambda_n)} \right| |X_n(x)| + \sum_{|n|=N}^\infty |f_n| | X_n(x)|
=I_N + R_N.
\end{align*} 
Since $f\in\mathcal U_{M,\gamma}$ with $\gamma>2-\alpha$, we see that
\begin{align*}
R_N = \sum_{|n|=N}^\infty |n^\gamma f_n| |n|^{-\gamma}| X_n(x)| 
\le M\sum_{|n|=N}^\infty |n|^{-\gamma}| X_n(x)|,
\end{align*}
from which we further use the property of the Mittag--Leffler function to derive
\begin{align*}
R_N \le& M\sum_{|n|=N}^\infty |n|^{-\gamma} \frac{Cx^{\alpha-1}}{(1+|\lambda_n|x^\alpha)^2}
\\
\le& M\sum_{|n|=N}^\infty |n|^{-\gamma} \frac{C(|\lambda_n| x^\alpha)^{1-\frac1{\alpha}}}{(1+|\lambda_n|x^\alpha)^2} |\lambda_n|^{\frac1\alpha-1}.
\end{align*}
Again by the use of the estimate \eqref{esti1} and the asymptotic estimate of the eigenvalue $\lambda_n$ we obtain
\begin{align*}
R_N \le C\sum_{|n|=N}^\infty |n|^{-\gamma} |\lambda_n|^{\frac1\alpha-1}
\le C\sum_{|n|=N}^\infty \left(\frac{1}{|n|} \right)^{\gamma+\alpha-1}.
\end{align*}
Finally, noting $\gamma>2-\alpha$ implies $\gamma+\alpha>2$, we conclude that 
\begin{align*}
R_N \le  C\sum_{|n|=N}^\infty \left(\frac{1}{|n|} \right)^{\frac{\gamma+\alpha}2} \left(\frac{1}{|n|} \right)^{\frac{\gamma+\alpha}2-1} 
\le CN^{1-\frac{\gamma+\alpha}2}. 
\end{align*}
Next we turn to the esimate for $I_N$. From Lemma \ref{lem-XnYn} and the asymptotic property of the Mittag--Leffler functions in \eqref{eq-asymp}, it follows that 
\begin{align*}
I_N \le& C\sum_{|n|=1}^{N-1}  |\lambda_n|^5 \left| \langle u_x(1,t), Z_n \rangle_{L_t^2(0,1)} \right| \frac{x^{\alpha-1}}{(1+|\lambda_n|x^\alpha)^2} 
\\
\le& C\sum_{|n|=1}^{N-1}  |\lambda_n|^5 \left| \langle u_x(1,t), Z_n \rangle_{L_t^2(0,1)} \right| \frac{C(|\lambda_n| x^\alpha)^{1-\frac1{\alpha}}}{(1+|\lambda_n|x^\alpha)^2} |\lambda_n|^{\frac1\alpha-1}
\\
\le& C\sum_{|n|=1}^{N-1}  |\lambda_n|^5 \left| \langle u_x(1,t), Z_n \rangle_{L_t^2(0,1)} \right| |\lambda_n|^{\frac1\alpha-1}.
\end{align*}
Now using the results in Theorem \ref{thm-fp}, we see that $|u_x(1,t)| \le C$ for any $t\in[0,T]$, which allows us to further estimate $I_N$ by
\begin{align*}
I_N \le C\|u_x(1,\cdot)\|_{C[0,T]}\sum_{|n|=1}^{N-1}  |\lambda_n|^{4+\frac1\alpha} \int_0^1 |Z_n(t)| dt.
\end{align*}
Moreover, from the definition of $Z_n(t)$, we see that
$$
\|Z_n\|_{L_t^2(0,1)} \le C\|Y_n\|_{L_t^2(0,1)} \le C\|Y_n\|_{C[0,1]} = C\|X_n\|_{C[0,T]} \le C|\lambda_n|^{\frac1\alpha-1},
$$
whcih implies
\begin{align*}
I_N \le C\|u_x(1,\cdot)\|_{C[0,1]}\sum_{|n|=1}^{N-1}  |\lambda_n|^{3+\frac2\alpha}.
\end{align*}
Again by noting the asymptotic estimate of $\lambda_n$, we further get the following estimate
\begin{align*}
I_N \le C\|u_x(1,\cdot)\|_{C[0,1]}\sum_{|n|=1}^{N-1} |n|^{3\alpha+2} \le C\|u_x(1,\cdot)\|_{C[0,1]}N^{3\alpha+3}.
\end{align*}
Letting $\|u_x(1,\cdot)\|_{C[0,1]} N^{3\alpha+3} = \left( \frac1N\right)^{\frac{\gamma+\alpha}2 - 1}$, that is, $N=\left( \frac1{\|u_x(1,\cdot)\|_{C[0,1]}}\right)^{\frac1{3\alpha+2+\frac{\alpha+\gamma}{2}}}$, finally we get
$$
\|f\|_{C[0,1]} \le C\|u_x(1,\cdot)\|_{C[0,1]}^{\frac{\frac{\alpha+\gamma}2 -1 }{3\alpha+2+\frac{\alpha+\gamma}2}},
$$
where $\alpha+\gamma>2$. By setting $\theta:=\frac{\frac{\alpha+\gamma}2 -1 }{3\alpha+2+\frac{\alpha+\gamma}2}$, we finish the proof of the theorem.
\end{proof}

\section{Numerical scheme and experiments}\label{sec-num}
In this section, we perform several numerical experiments for the reconstruction of the unknown source term $f$ in the domain $\Omega=(0,1)$ from the observed data $u_x(1,t)$ for $t\in(0,T)$.

For the direct problem, we discretize the operator $D_x^\beta$ on a graded mesh grid over $[0,1]$
with the Dirichlet boundary conditions using the piecewise linear finite element.
The graded mesh is defined as
\[ x_i=(ih)^\gamma, \quad i=0,\cdots,N, \]
where $h=\frac1N$ is the mesh size parameter and $\gamma>1$ is the grading parameter to cluster nodes near $x=0$, where the solution to the direct problem becomes more singular.
Meanwhile, the operator $\partial_t^\alpha$ is discretized on a temporal mesh grid with time step size $\tau>0$.
For the inverse problem, as the observed data $u_x(1,\cdot)$ linearly depend on the unknown source term $f$, we can formulate the problem as
\[ A f^h = z^{\delta,\tau}, \]
where $f^h$ is the discretization of $f$ on the spatial mesh grid and $z^{\delta,\tau}$ is the discretization of $u_x(1,\cdot)$ over the temporal mesh grid, added relative noise uniformly distributed in $[-\delta,\delta]$.
The matrix $A$ is ill-conditioned due to the ill-posedness of the inverse problem. We apply the Tikhonov regularization to solve the inverse problem:
\begin{align*}
    \hat{f}^h = (A^TA+\nu I)^{-1}A^T z^{\delta,\tau},
\end{align*}
with a regularization parameter $\nu>0$.
In the following numerical experiments, we solve the direct problem on a fine grid with $h=10^{-3}$ and $\tau=5\times10^{-4}$ and the inverse problem on a coarse grid with $h=4\times10^{-3}$ and $\tau=2\times10^{-3}$. Unless otherwise stated, we take the length of the observation time interval $T=1$ and the grading parameter $\gamma=4$ for spatial mesh grids.

\begin{example}\label{example:1}
In the first experiment, we take $\alpha=\beta=1.5$ and $\lambda(t)=2e^t$ for $t\in[0,1]$. We reconstruct the source term $f(x)=x(1-x)$, $f(x)=x^2(1-x)$, and $f(x)=x^4(1-x)$ from their respective observational data using our method, see Figure~\ref{fig:Ex1}. The noise levels are chosen to be $\delta=2\%$ and $\delta=5\%$.
\end{example}
We observe in Figure~\ref{fig:Ex1} that, as higher-order derivatives of $f$ vanish at $x=0$, the reconstruction becomes more accurate. In addition, increasing the noise level from $\delta=2\%$ to $\delta=5\%$ has minimal impact on the reconstruction.

\begin{figure}[htbp]
\centering\setlength{\tabcolsep}{0pt}
\begin{tabular}{ccc}
\includegraphics[width=.32\textwidth]{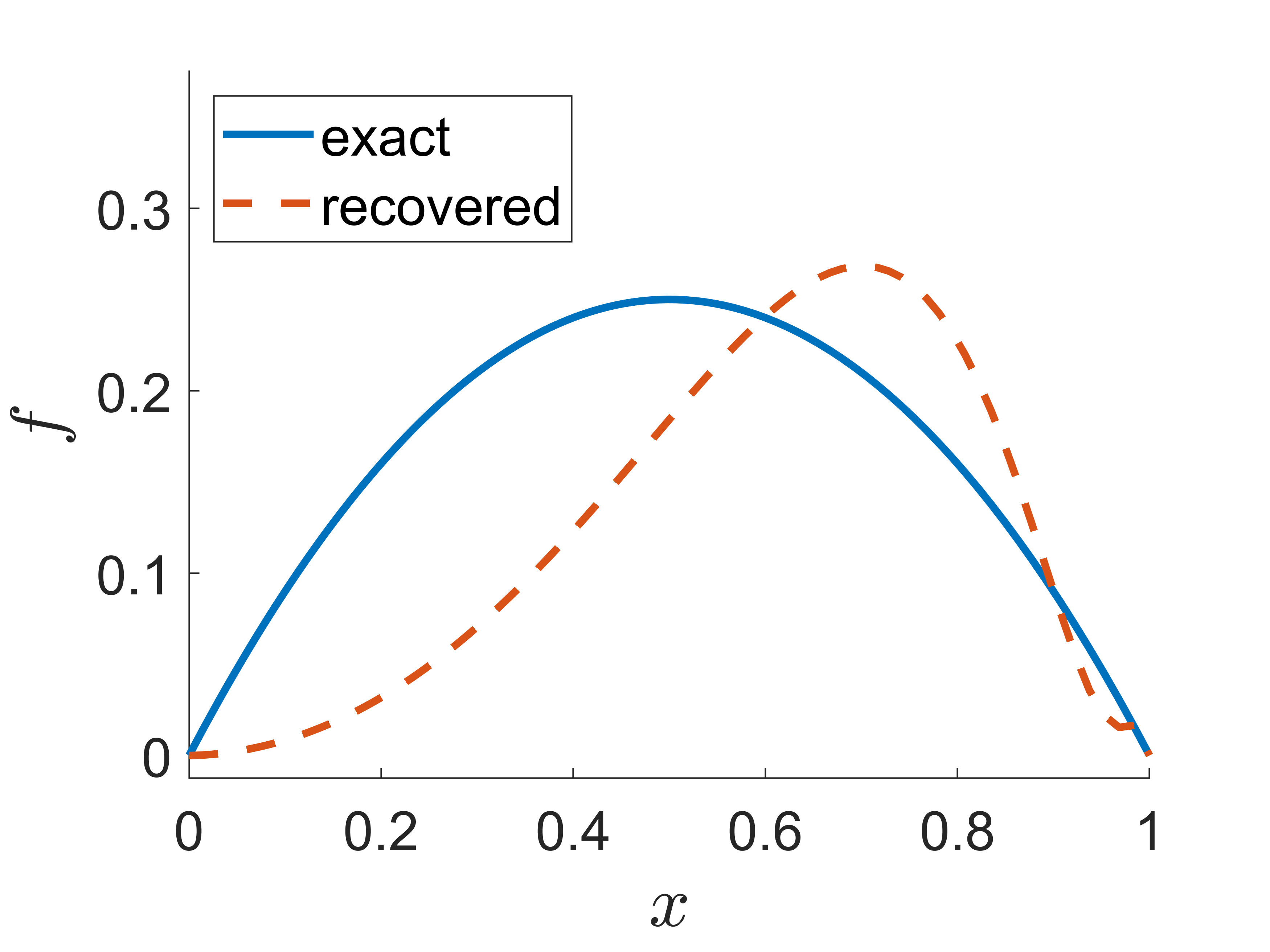} & \includegraphics[width=.32\textwidth]{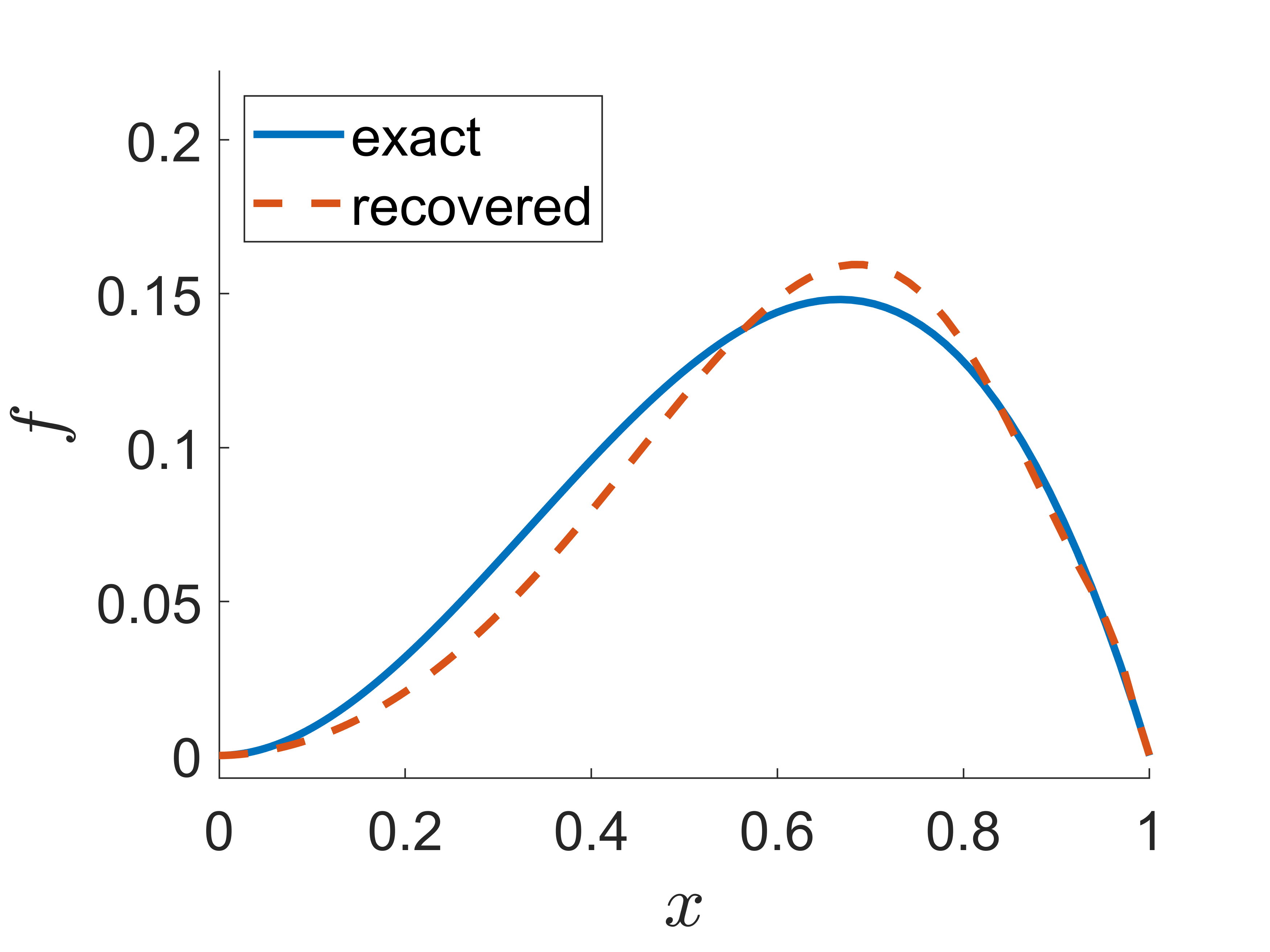} & \includegraphics[width=.32\textwidth]{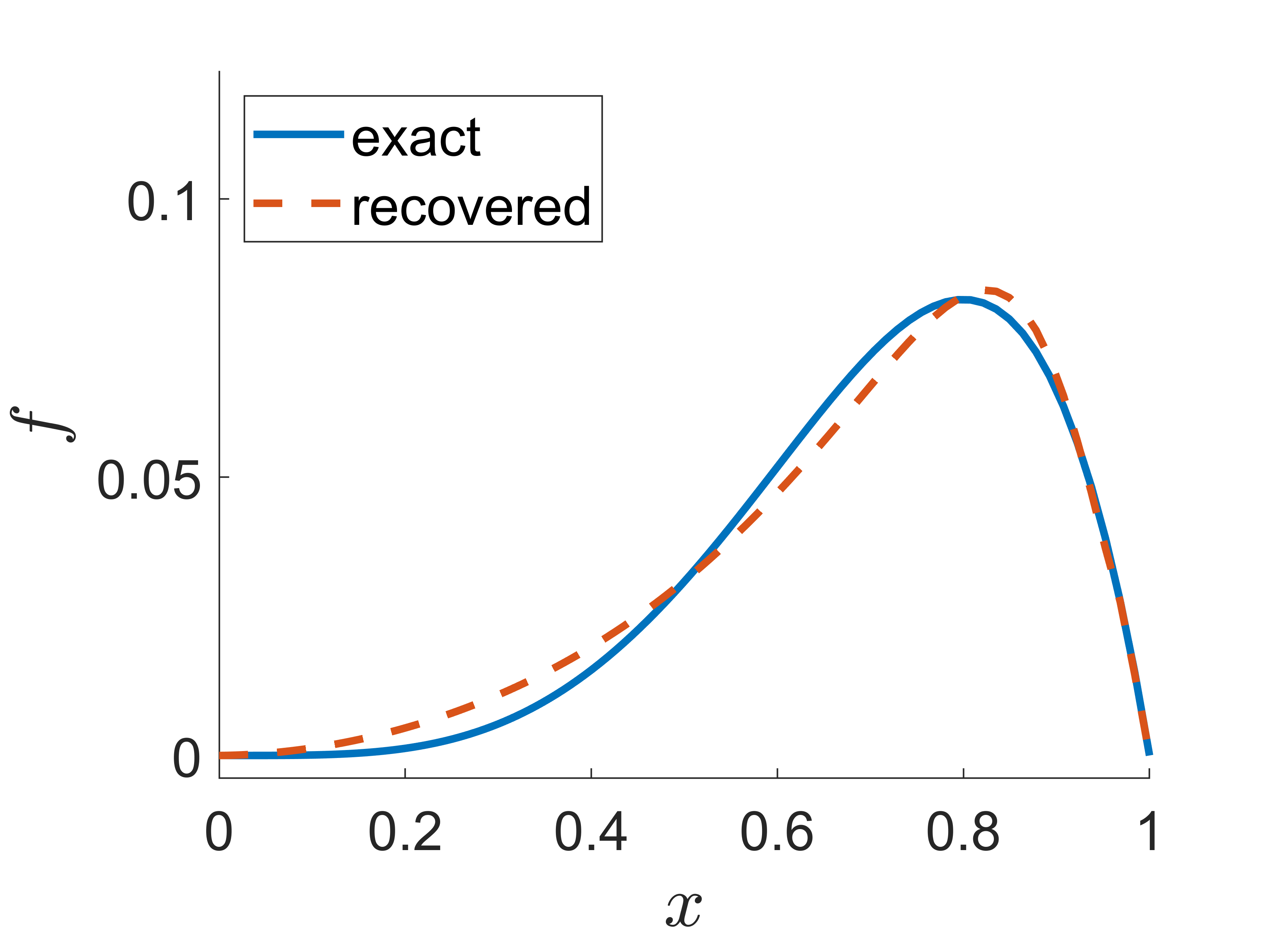} \\
\includegraphics[width=.32\textwidth]{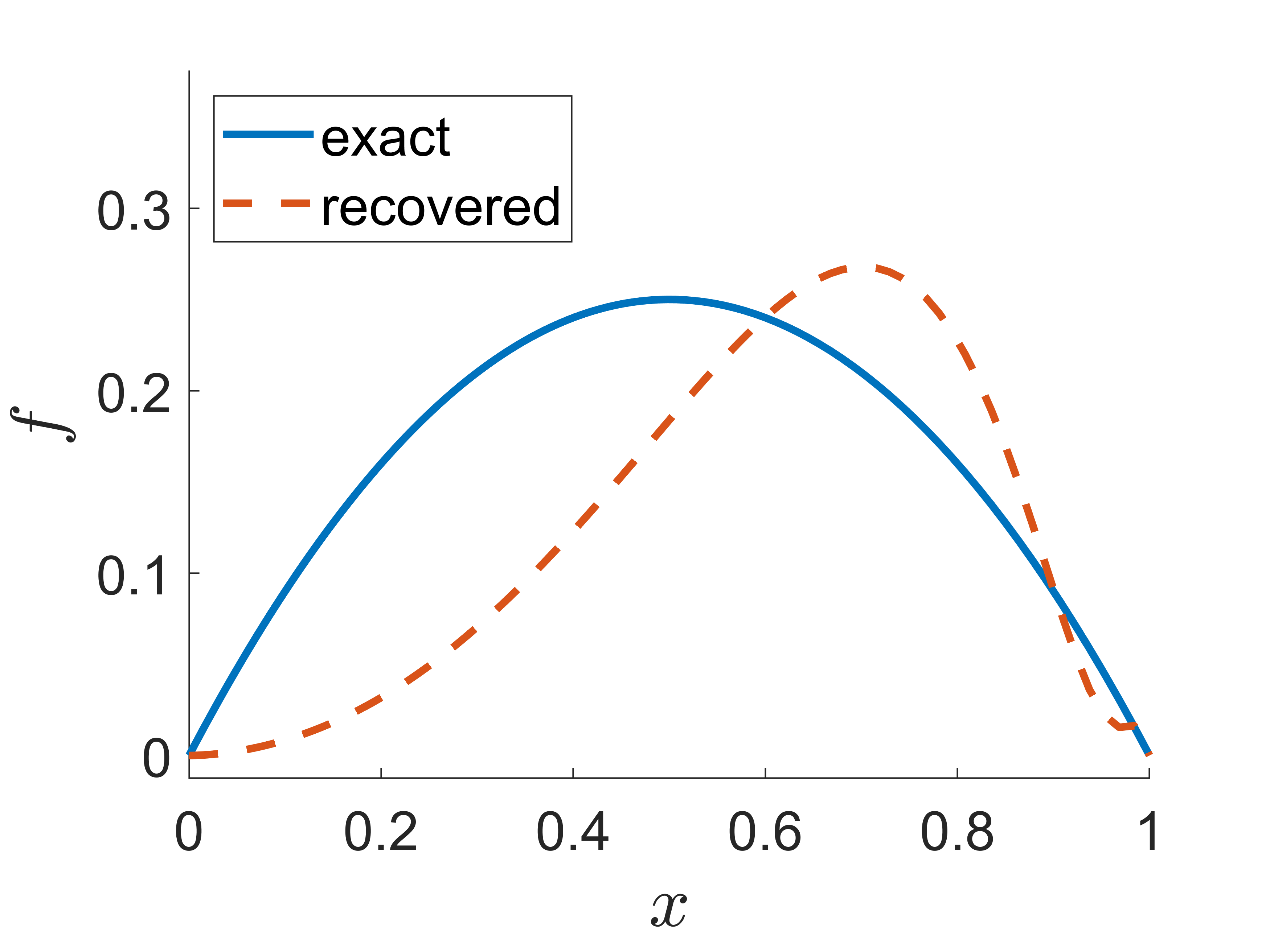} & \includegraphics[width=.32\textwidth]{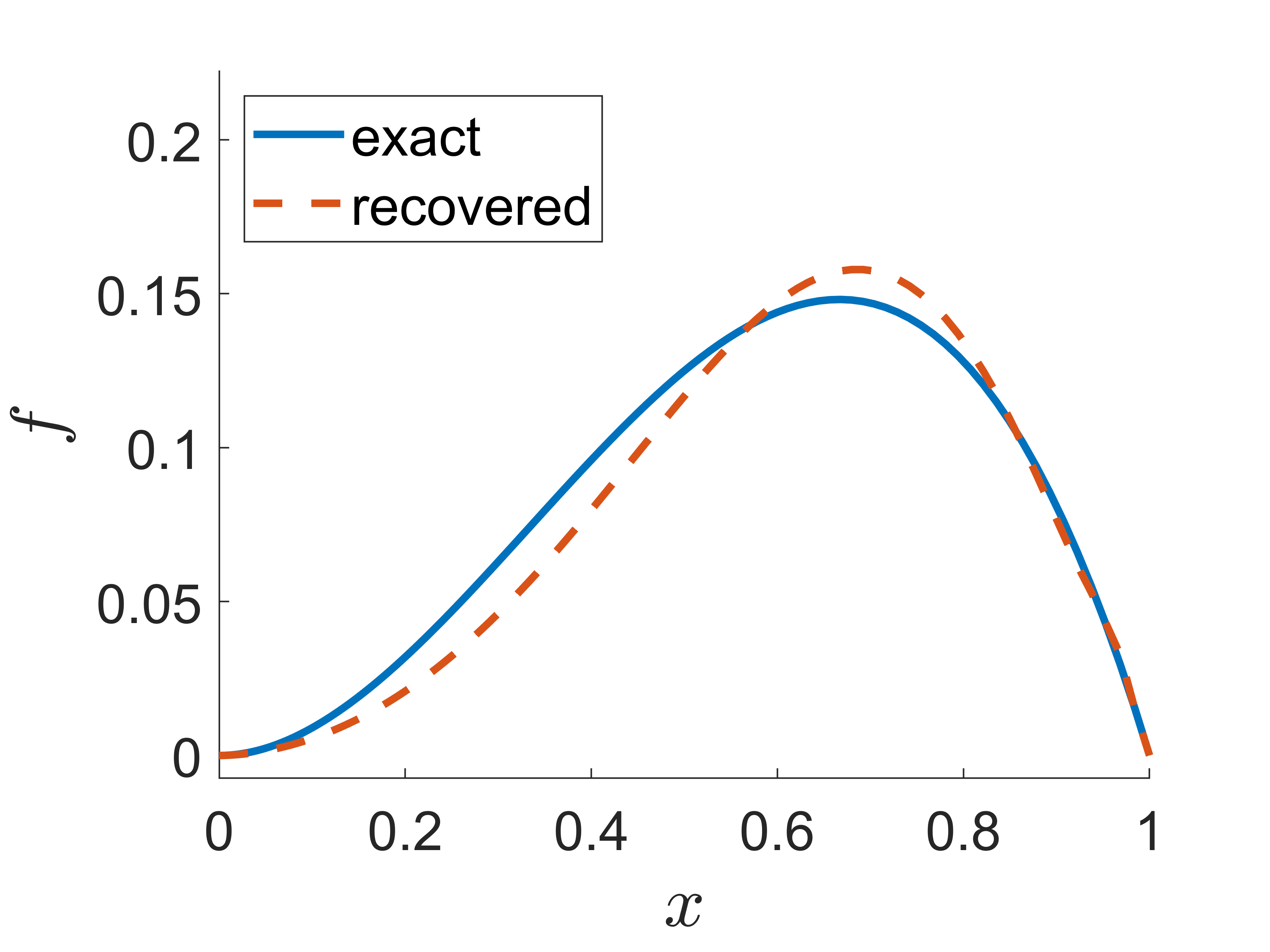} & \includegraphics[width=.32\textwidth]{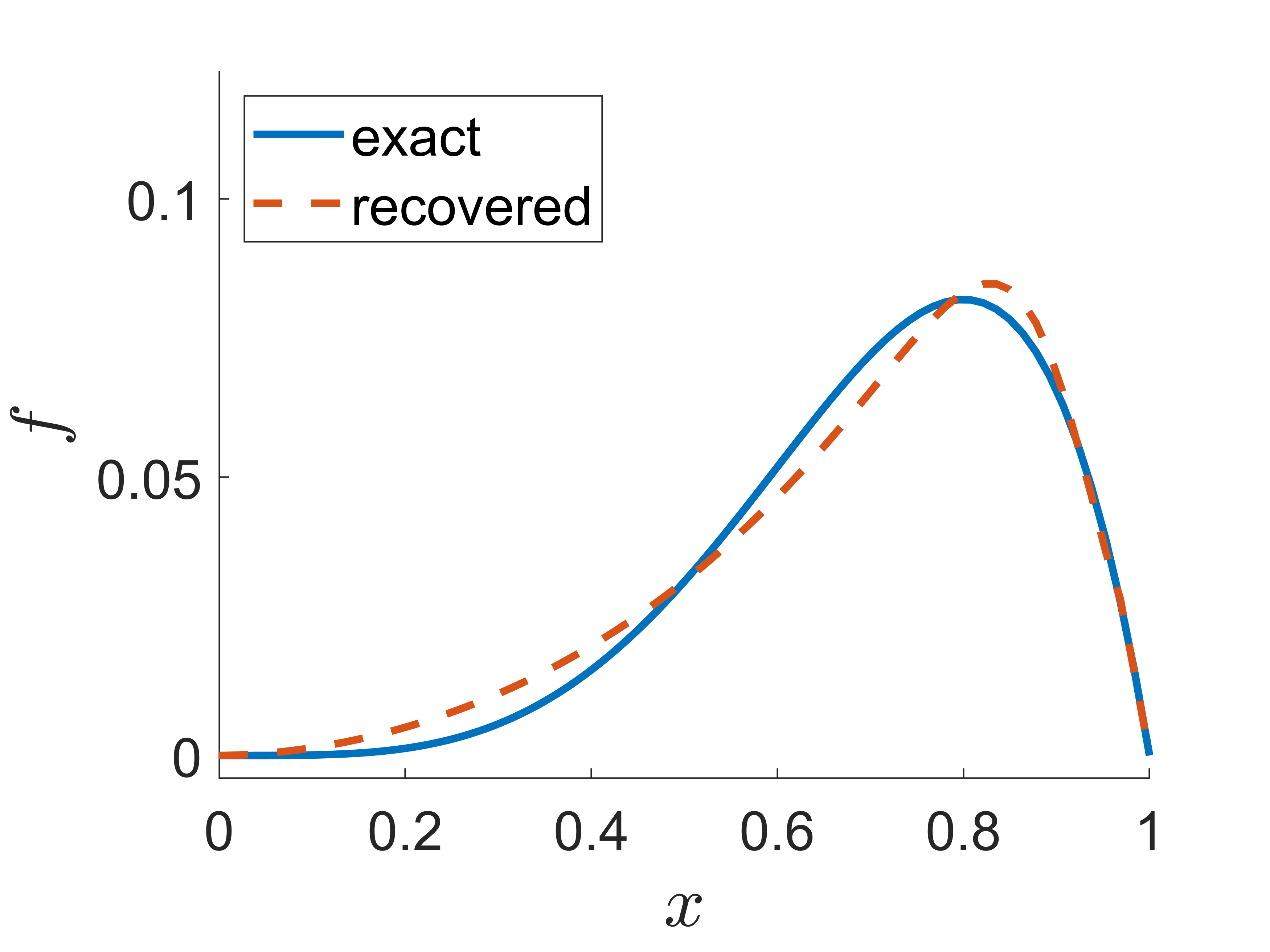} \\
(i) $f(x)=x(1-x)$ & (ii) $f(x)=x^2(1-x)$ & (iii) $f(x)=x^4(1-x)$  
\end{tabular}
\caption{Numerical results for Example~\ref{example:1}. Top row: $\delta=2\%$; Bottom row: $\delta=5\%$.}
\label{fig:Ex1}
\end{figure}

\begin{example}\label{example:2}
In the second experiment, we examine the effect of different source intensities and fractional orders. We reconstruct the source term $f(x)=x^2(1-x)$ using both $\lambda(t)=2e^t$ and $\lambda(t)=5\sin t$ for $t\in[0,1]$. For each source intensity, we take the fractional order to be $\alpha=1.2$, $\alpha=1.6$, and $\alpha=2$, all with $\beta=\alpha$. The inverse problem is solved with a noise level $\delta=2\%$ using our method.
\end{example}
We observe from Figure~\ref{fig:Ex2} that the quality of reconstruction increases as the fractional order $\alpha=\beta$ increases. When $\alpha$ is close to one, the inverse problem is severely ill-posedness thus the reconstruction is inaccurate and oscillatory. As $\alpha$ increases, the reconstruction achieves improved accuracy and smoothness, and the choice of source intensity has little impact on the reconstruction. When $\alpha=2$, the direct problem becomes a standard wave equation, the reconstruction is very accurate regardless of the choice of the source intensity. This observation is consistent with our theoretical investigation.

\begin{figure}[htbp]
\centering\setlength{\tabcolsep}{0pt}
\begin{tabular}{ccc}
\includegraphics[width=.32\textwidth]{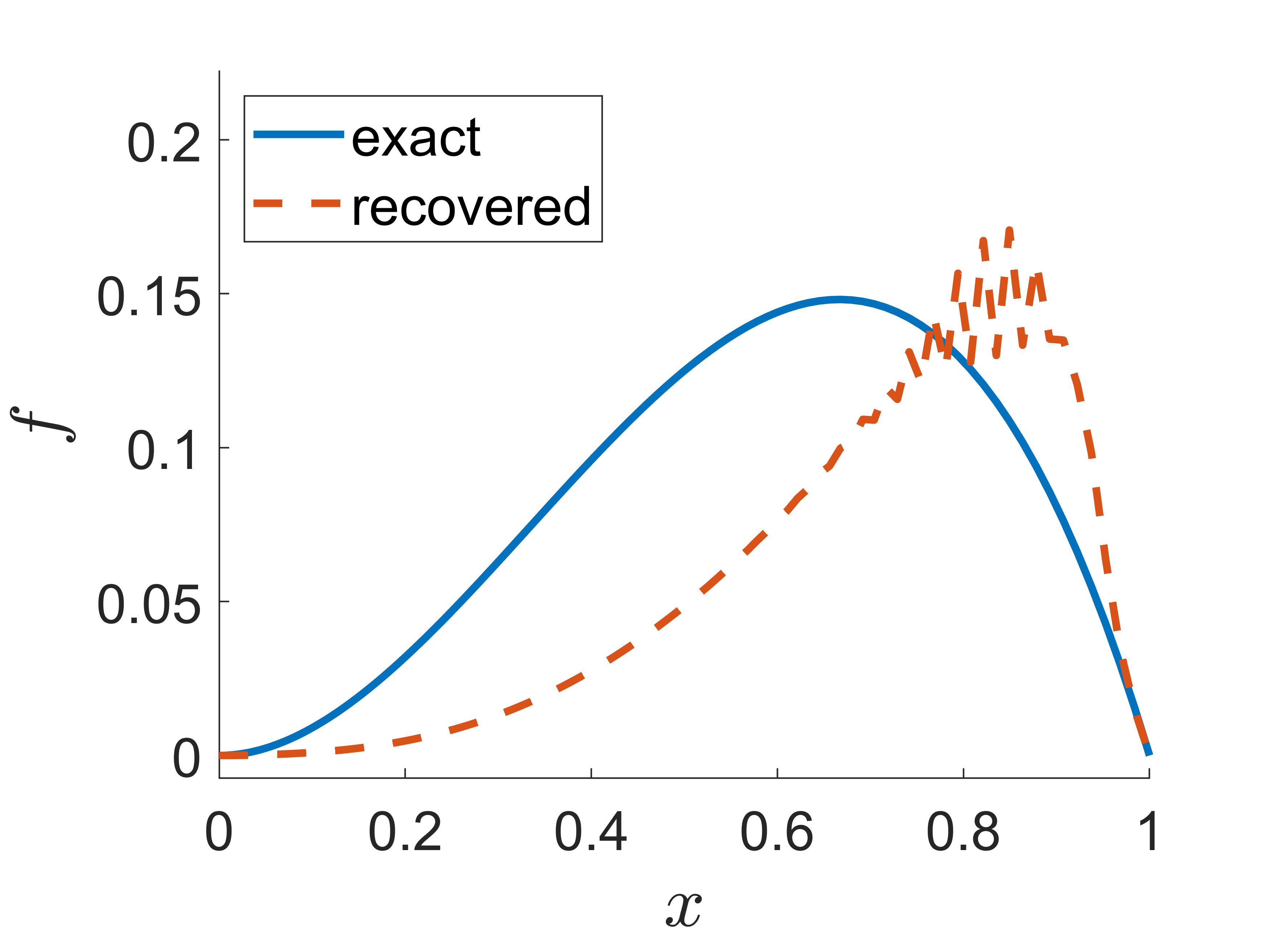} & \includegraphics[width=.32\textwidth]{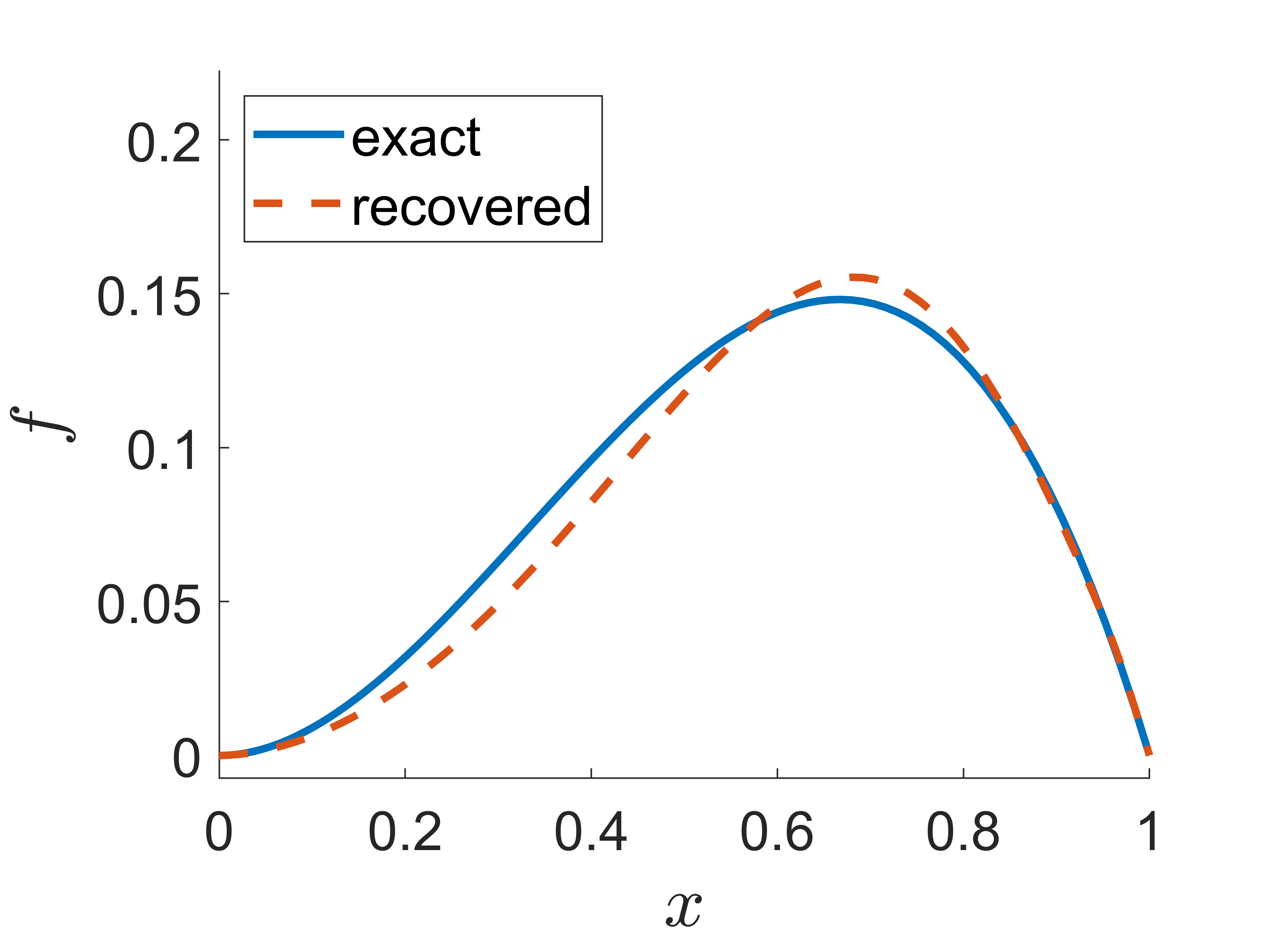} & \includegraphics[width=.32\textwidth]{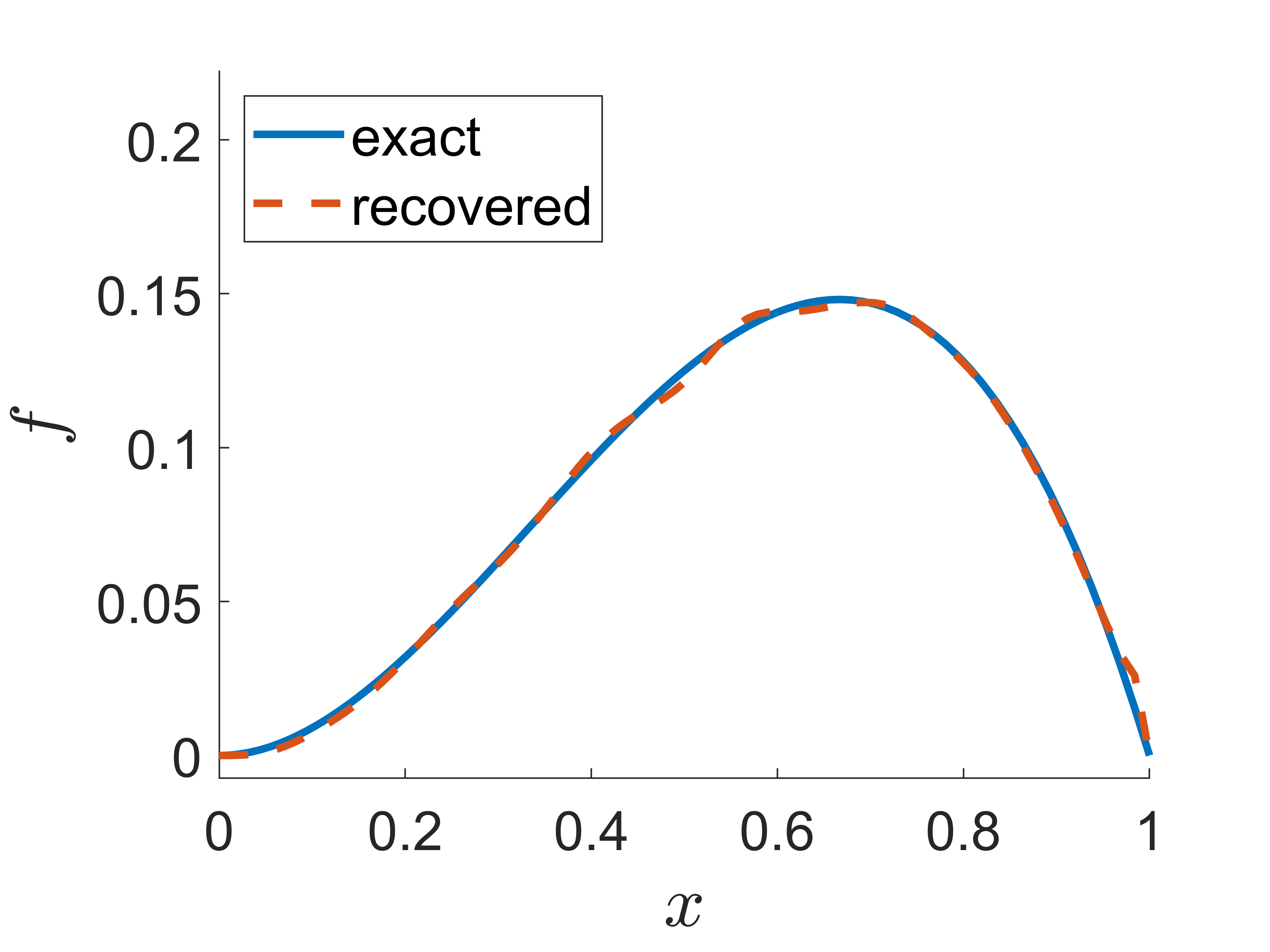} \\
\includegraphics[width=.32\textwidth]{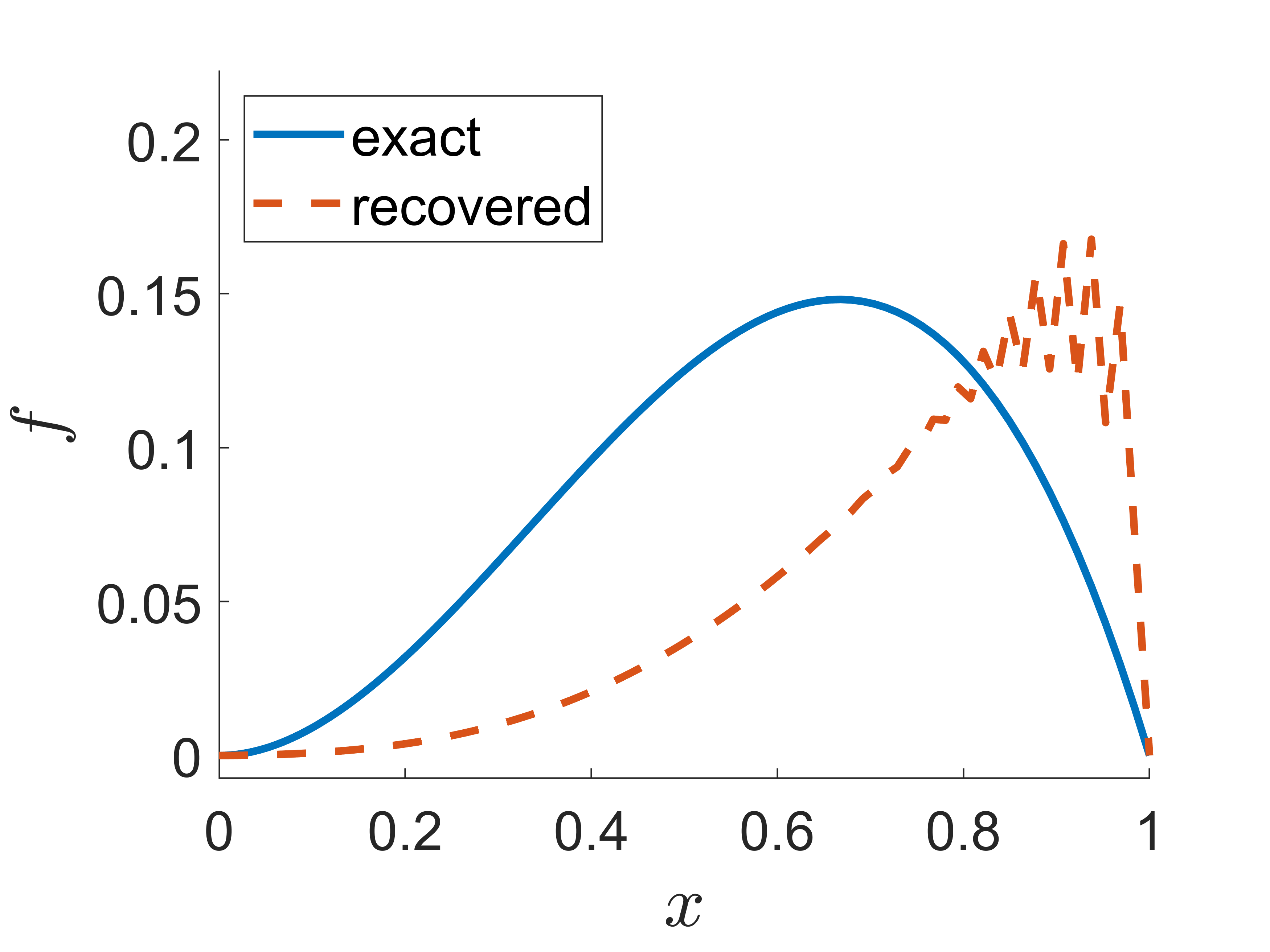} & \includegraphics[width=.32\textwidth]{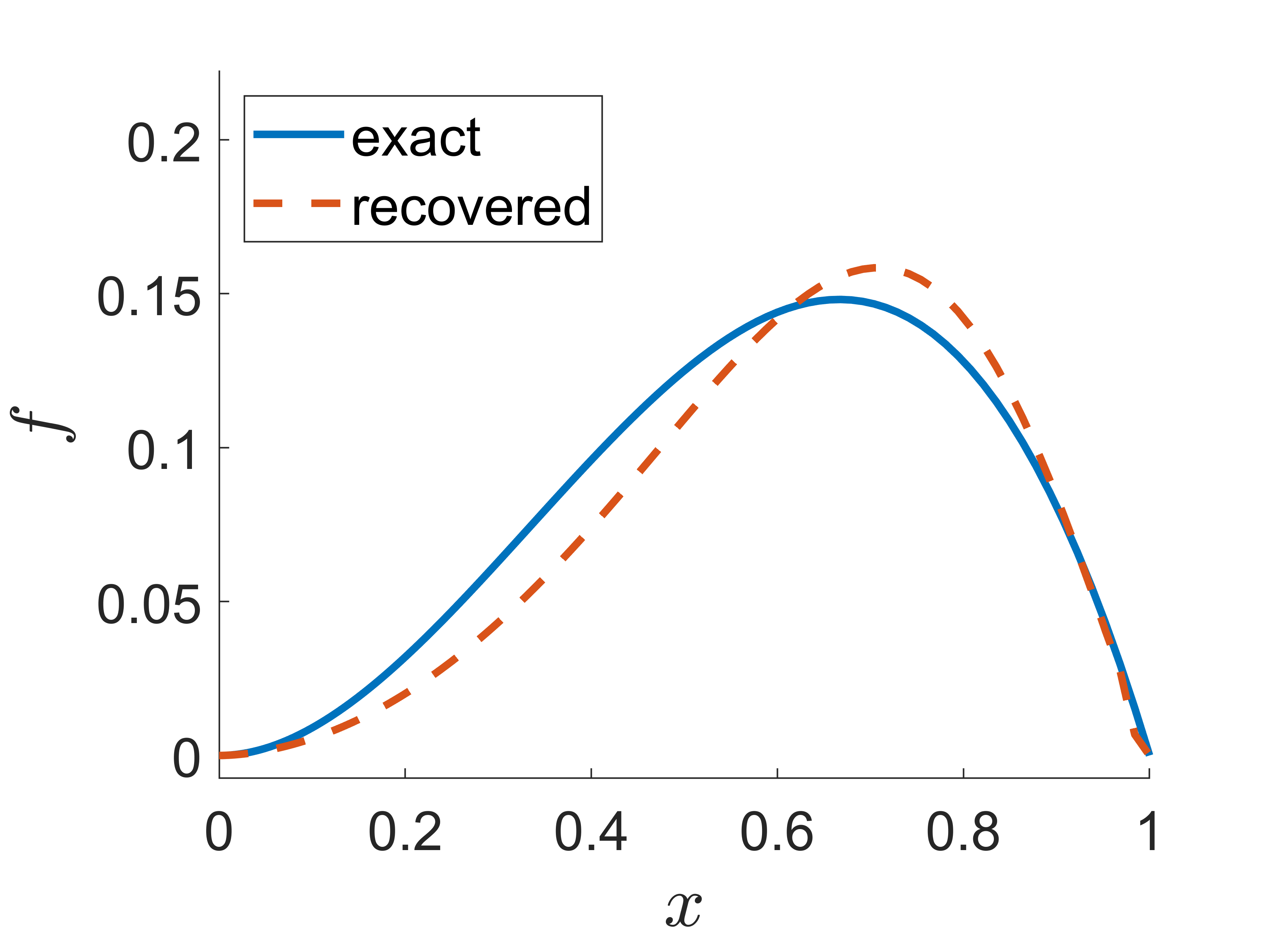} & \includegraphics[width=.32\textwidth]{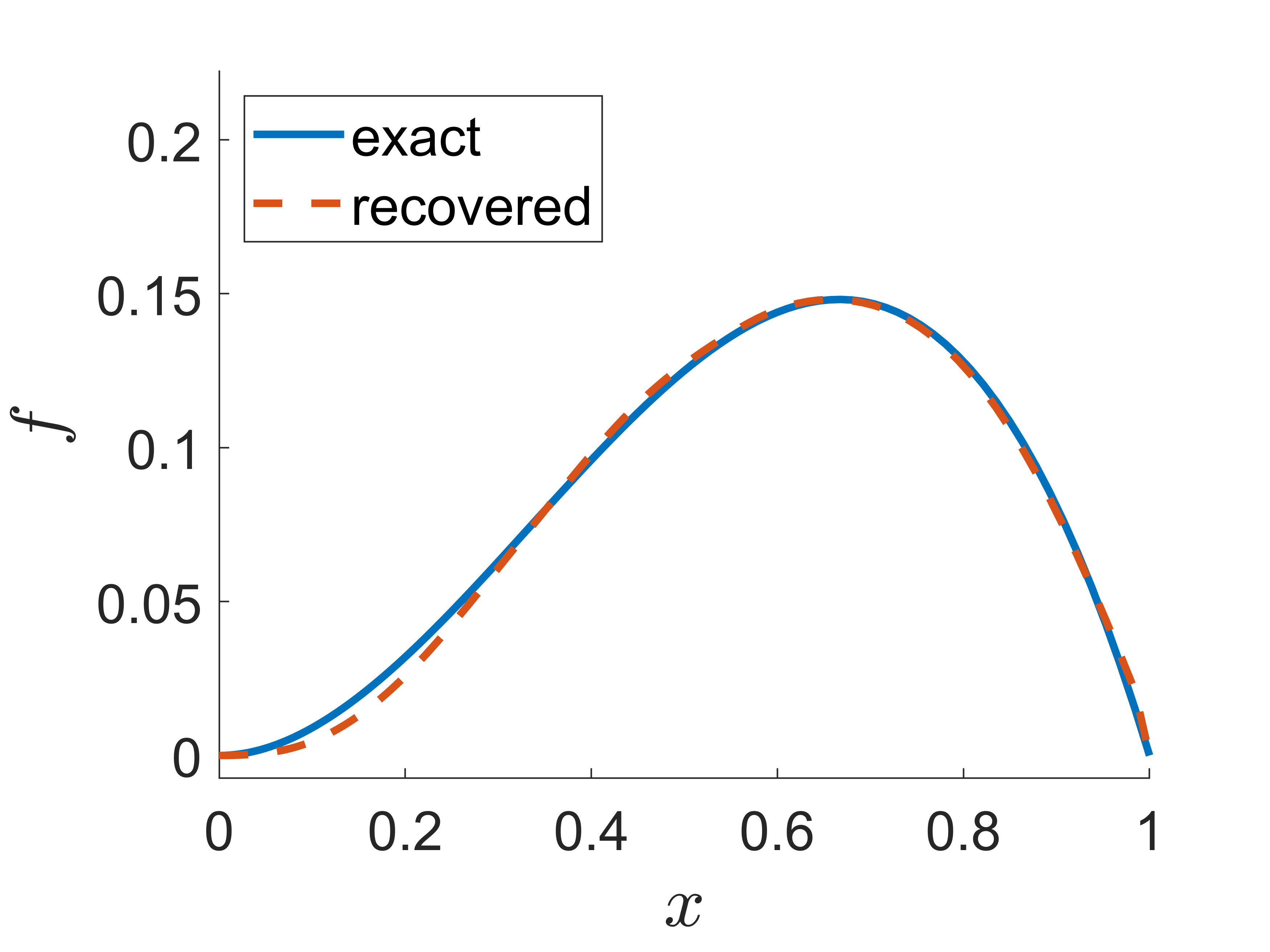} \\
$\alpha=1.2$ & $\alpha=1.6$ & $\alpha=2$  
\end{tabular}
\caption{Numerical results for Example~\ref{example:2}. Top row: $\lambda(t)=2e^t$; Bottom row: $\lambda(t)=5\sin t$.}
\label{fig:Ex2}
\end{figure}

\begin{example}\label{example:3}
In the third experiment, we consider the regime where $\alpha\neq\beta$. We examine two cases where $\alpha=1.2$ and $\alpha=1.4$, respectively, and in each case we take $\beta=1.3$, $\beta=1.5$, and $\beta=1.7$. We reconstruct the source term $f(x)=x^2(1-x)$ with the source intensity $\lambda(t)=2e^t$ for $t\in[0,1]$. The inverse problem is solved with a noise level $\delta=2\%$ using our method.
\end{example}

\begin{figure}[htbp]
\centering\setlength{\tabcolsep}{0pt}
\begin{tabular}{ccc}
\includegraphics[width=.32\textwidth]{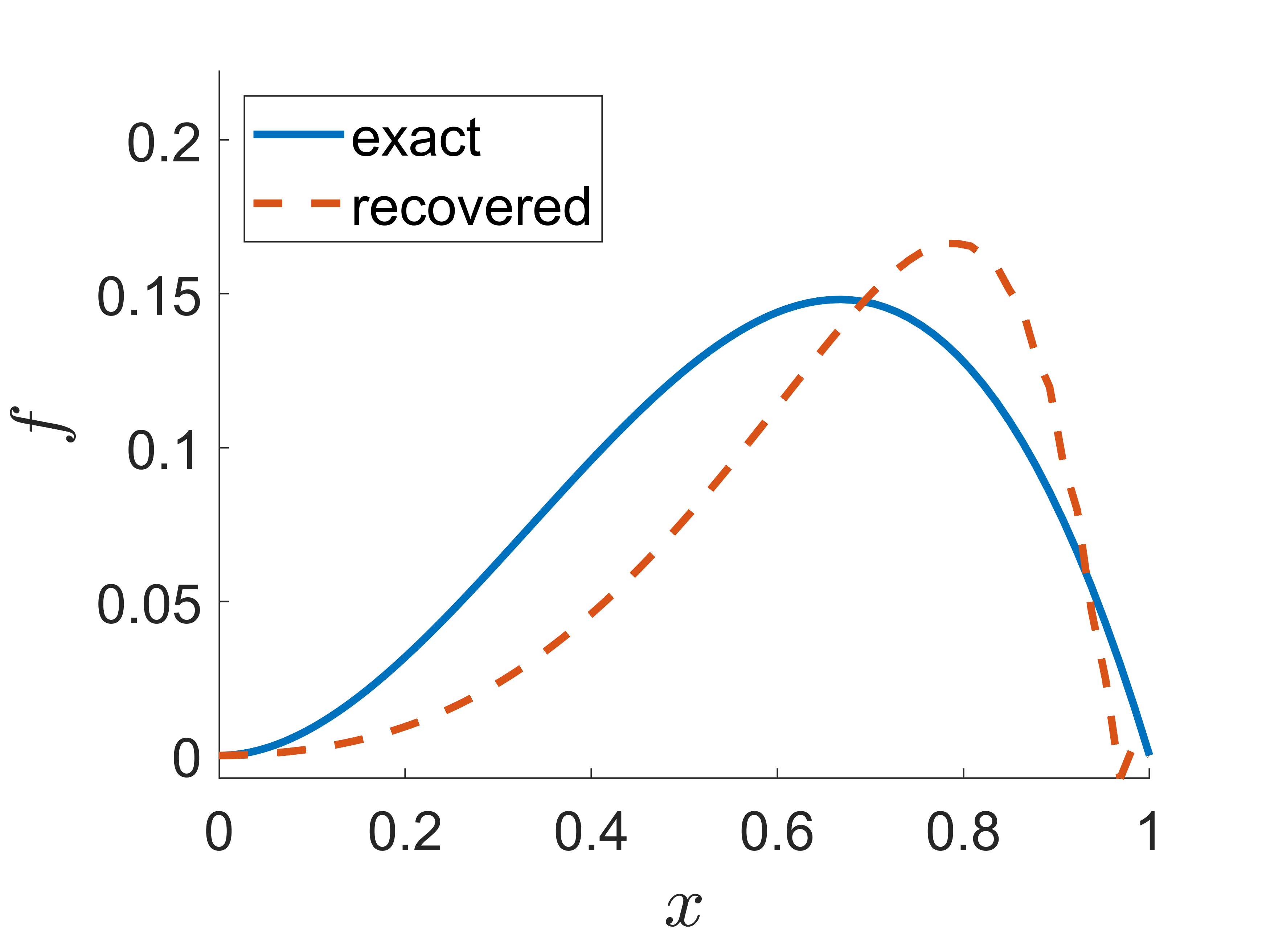} & \includegraphics[width=.32\textwidth]{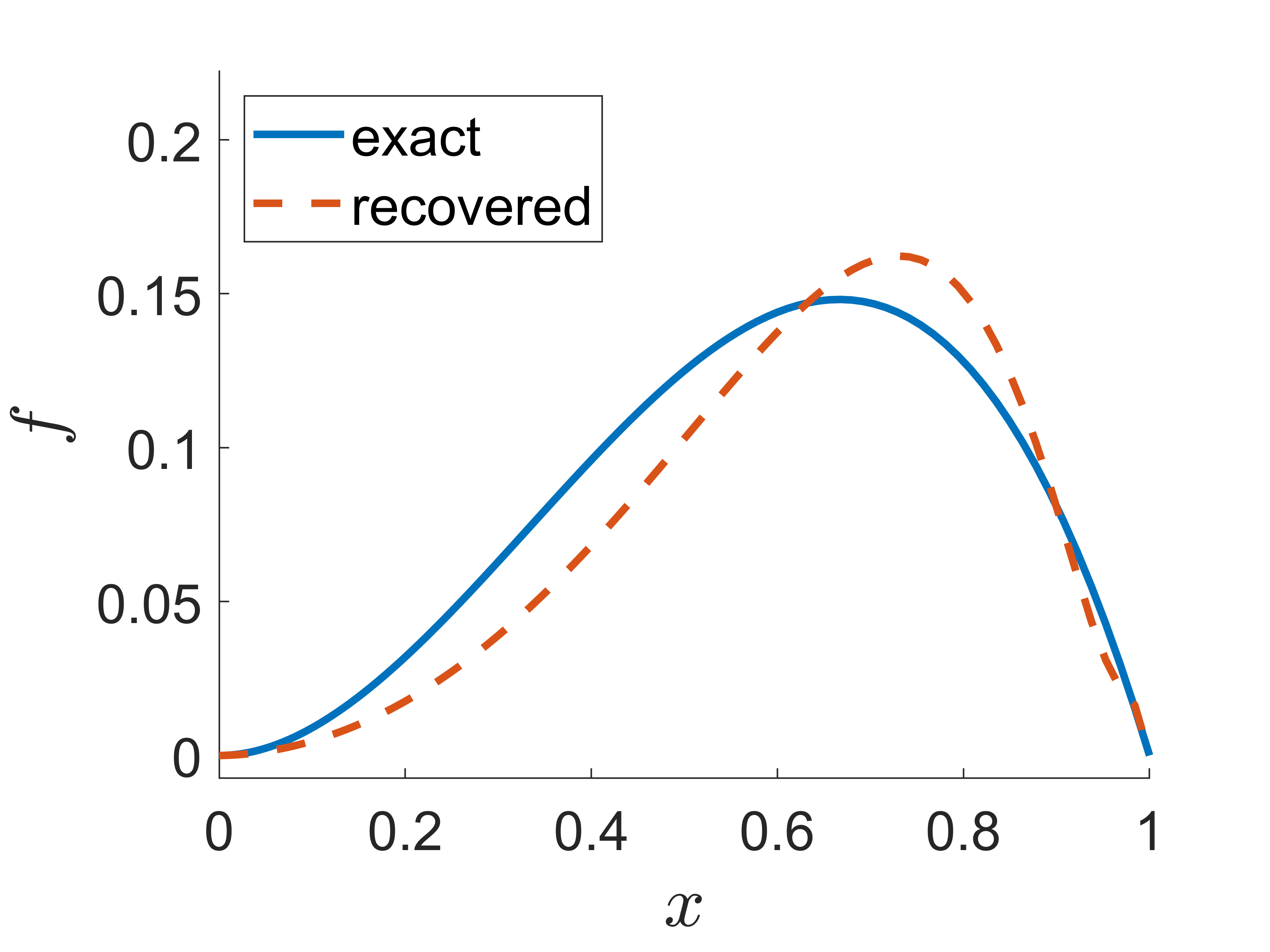} & \includegraphics[width=.32\textwidth]{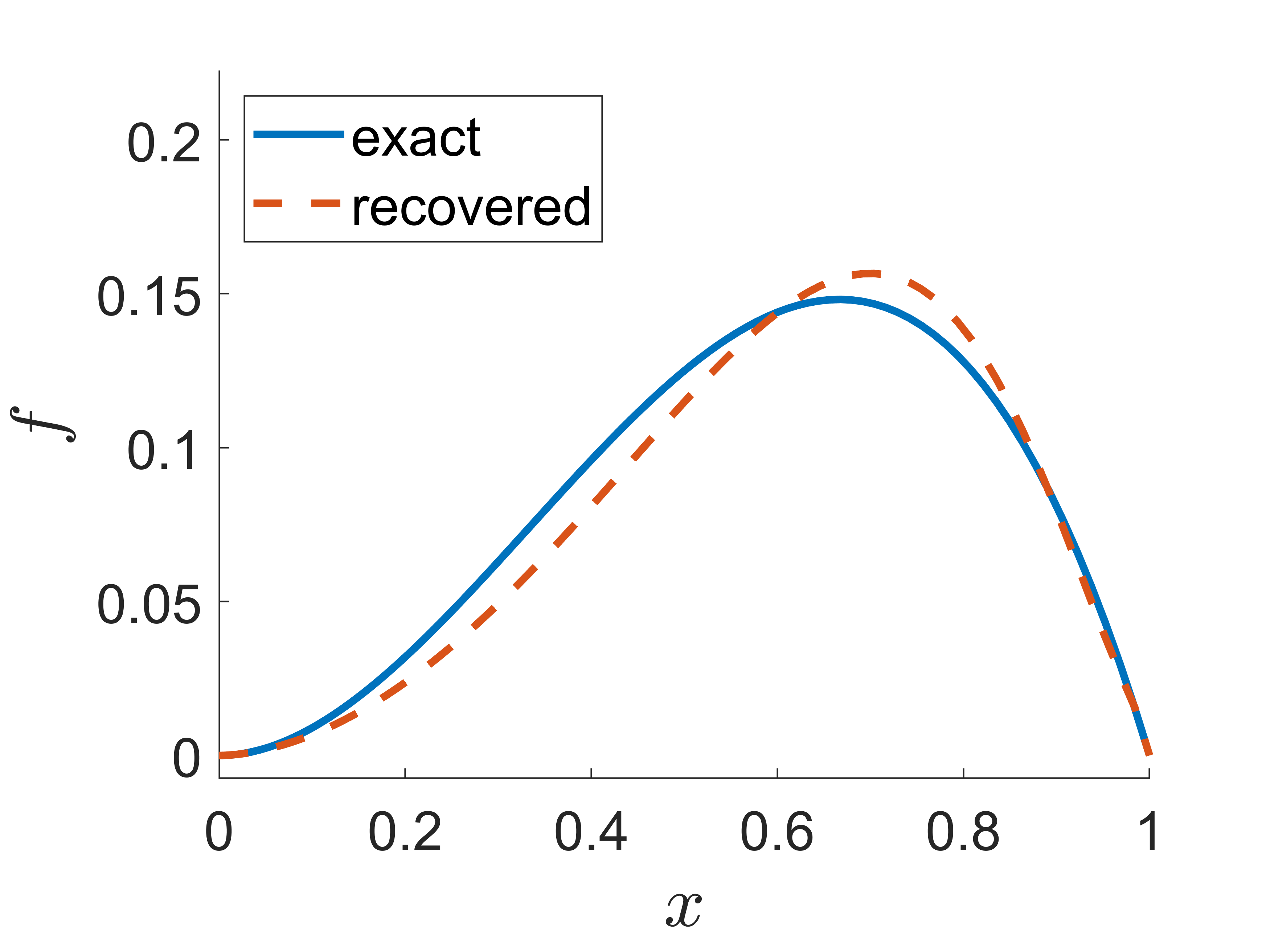} \\
\includegraphics[width=.32\textwidth]{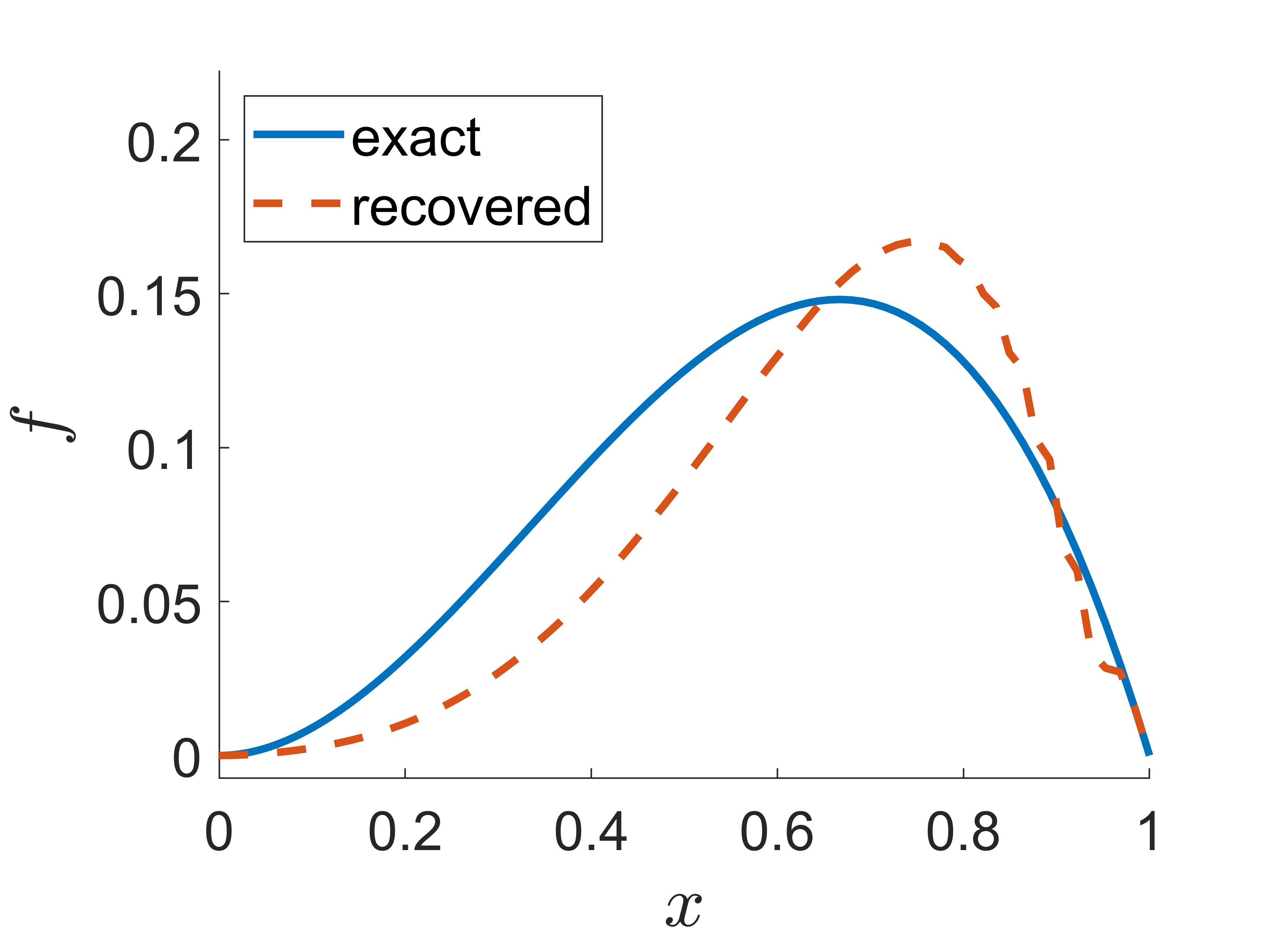} & \includegraphics[width=.32\textwidth]{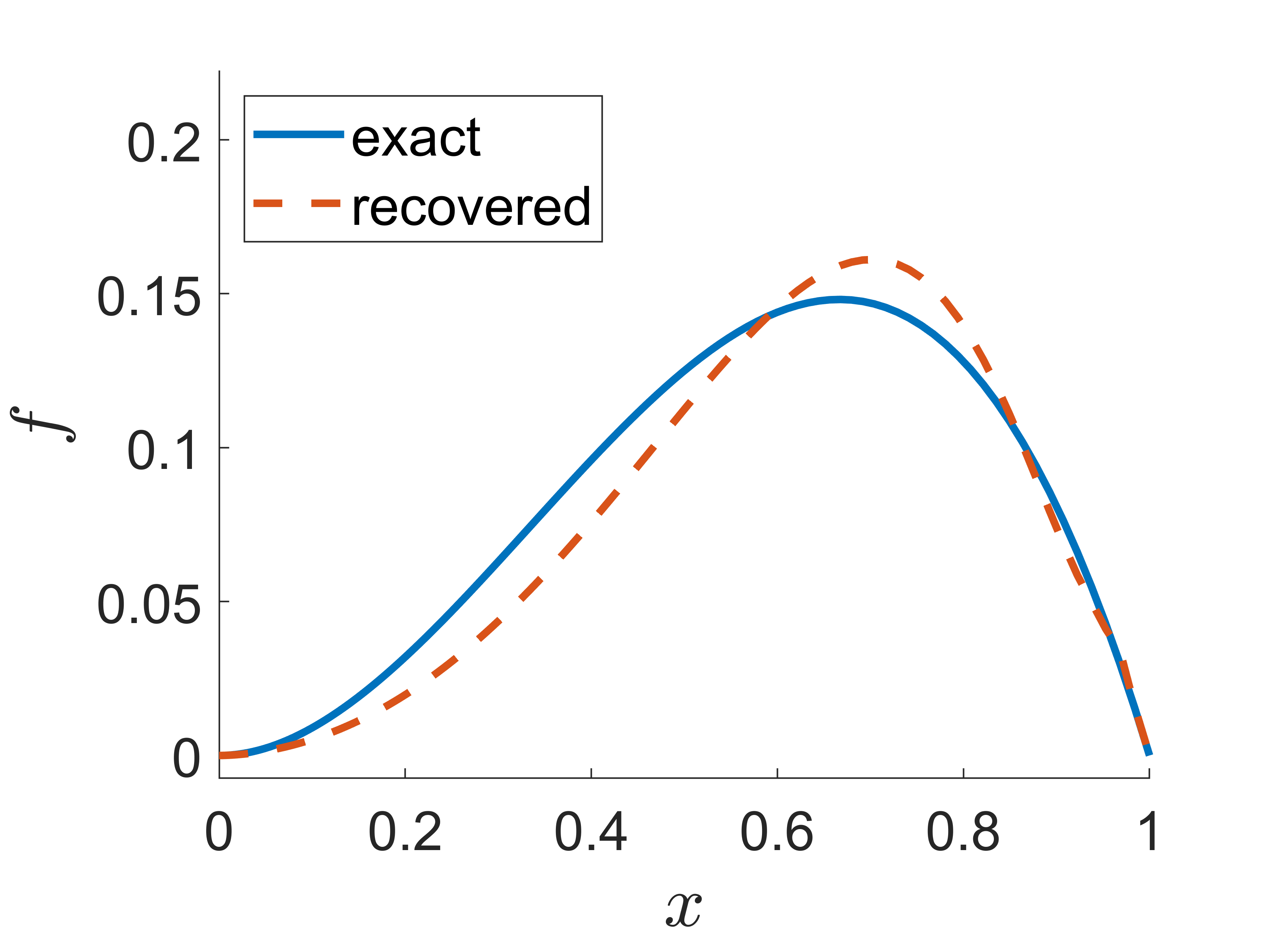} & \includegraphics[width=.32\textwidth]{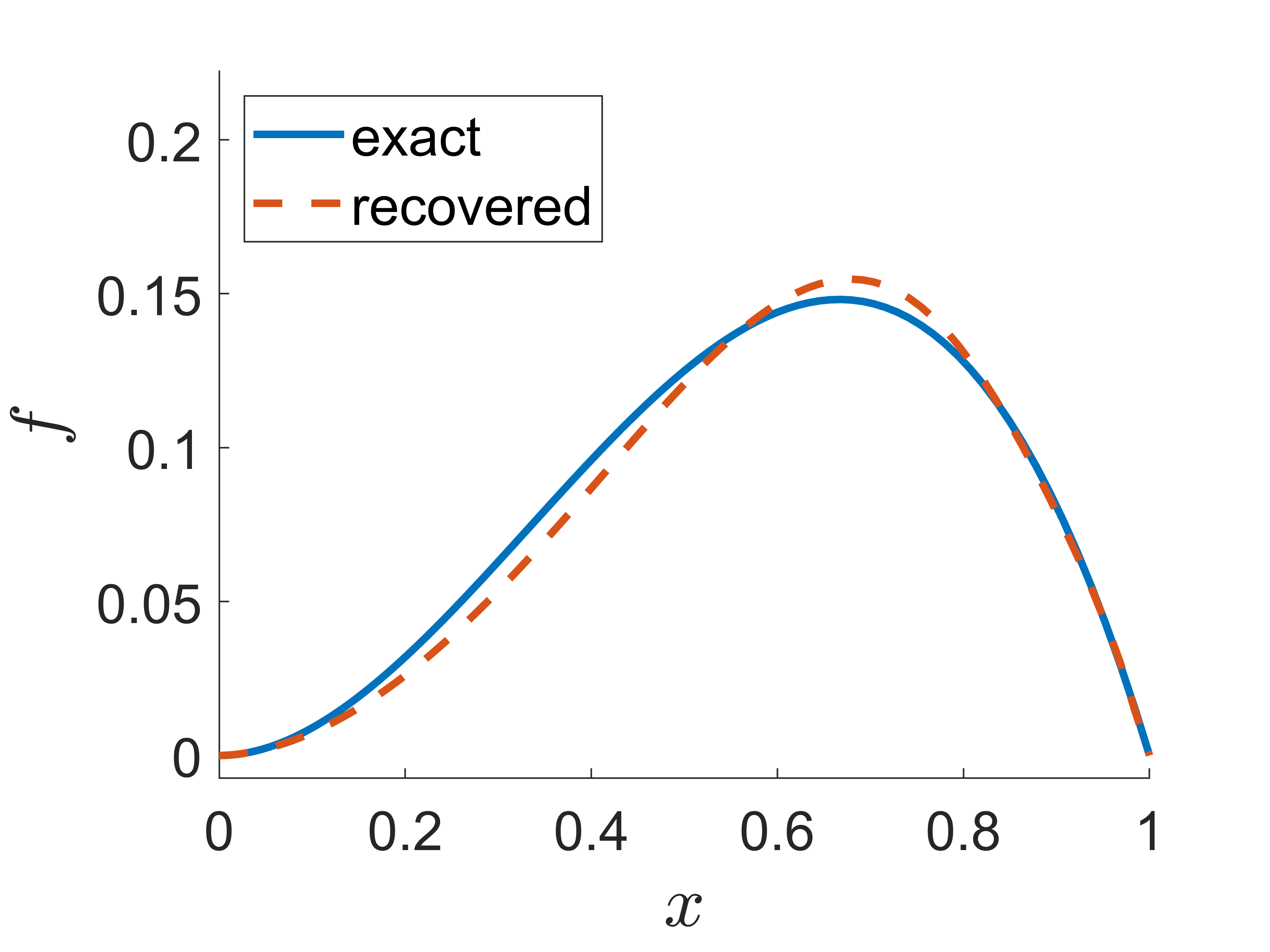} \\
$\beta=1.3$ & $\beta=1.5$ & $\beta=1.7$  
\end{tabular}
\caption{Numerical results for Example~\ref{example:3}. Top row: $\alpha=1.2$; Bottom row: $\alpha=1.4$.}
\label{fig:Ex3}
\end{figure}
We observe in Figure~\ref{fig:Ex3} that, in the regime where the spatial fractional derivative operator and the temporal one have different orders, our method still provides fairly accurate reconstructions. Moreover, similar to observations in Example~\ref{example:2}, for a fixed $\alpha$, the reconstruction becomes more accurate as $\beta$ increases due to stronger well-posedness of the direct problem.
These numerical results suggest that our theoretical results may still hold for general fractional orders $1<\alpha,\beta<2$, while stability of the inverse problem highly depends on the choices of them.

\section{Conclusions}\label{sec-con}

In this paper, we investigate an inverse source problem for a one-dimensional space-time fractional wave equation. We used boundary measurement data at a single endpoint to recover the spatially dependent source term. The main challenge comes from the non-orthogonality of the eigenfunctions for the Dirichlet eigenvalue problem of the spatial fractional derivative. To overcome this, we introduced a bi-orthogonal basis associated with the Mittag--Leffler functions. Crucially, we established uniqueness results for general orders $\alpha,\beta\in (1,2)$, and conditional stability results (H\"older type) specifically when $\alpha=\beta$ lies in $(1,2)$, using the properties of this bi-orthogonal system. For the numerical solution of the inverse source problem, a Tikhonov regularization method was employed. Several numerical examples demonstrated the accuracy and efficiency of the proposed method and validated our theoretical findings.

However, it is important to note the following limitations and open problems.
\begin{enumerate}
    \item Dimensionality: Our analysis is confined to the one-dimensional case. The extension of these results to higher-dimensional fractional wave equations remains an open challenge.
    \item Order Restriction for Stability: The stability result relies on the bi-orthogonal basis constructed for the Mittag--Leffler functions only when $\alpha = \beta$. For other cases ($\alpha \neq \beta \in (1,2)$), establishing stability remains open due to the lack of a suitable bi-orthogonal system.
\end{enumerate}

\section*{Acknowledgments}
Zhidong Zhang is supported by the National Key Research and Development Plan of China (Grant No. 2023YFB3002400). 
Zhiyuan Li thanks the National Natural Science Foundation of China (12271277), and Ningbo Youth Leading Talent Project (2024QL045). This work is partially supported by the Open Research Fund of Key Laboratory of Nonlinear Analysis \& Applications (Central China Normal University), Ministry of Education, China. The work of Z. Zhou is supported by by National Natural Science Foundation of China (Project 12422117), Hong Kong Research Grants Council (15303021), and an internal grant of Hong Kong Polytechnic University (Project ID: P0038888, Work Programme: ZVX3).

\bibliographystyle{abbrv}
\bibliography{refs}

\end{document}